\pdfoutput=1

\documentclass[12pt]{amsart}
\usepackage[dvips]{graphicx}
\usepackage{color}
\usepackage{amssymb} 
\usepackage[english]{babel}
\usepackage{amsmath,amsthm}
\usepackage{amsfonts}
\usepackage{latexsym}
\usepackage{amscd}
\usepackage[latin1]{inputenc}
\usepackage{enumerate}
\usepackage{afterpage}
\usepackage[all]{xy}
\usepackage{mathrsfs}
\usepackage{array}

\newcommand{\C}{\mathbb{C}}

\newcommand{\Z}{\mathbb{Z}}

\def\GL{\mathop{\rm  GL}}

\renewcommand{\to}{\longrightarrow}

\newtheorem{Theorem}{Theorem}[section]
\newtheorem{Definition}[Theorem]{Definition}
\newtheorem{Lemma}[Theorem]{Lemma}
\newtheorem{Proposition}[Theorem]{Proposition}
\newtheorem{Corollary}[Theorem]{Corollary}
\newtheorem{Remark}[Theorem]{Remark}
\newtheorem{Example}[Theorem]{Example}

\newtheorem{introthm}{Theorem}

\DeclareMathOperator{\Der}{Der}
\DeclareMathOperator{\Ann}{Ann}

\DeclareMathOperator{\Aut}{Aut}
\DeclareMathOperator{\image}{image}
\DeclareMathOperator{\coker}{coker}
\DeclareMathOperator{\trace}{trace}

\DeclareMathOperator{\Spec}{Spec}
\DeclareMathOperator{\Gr}{Gr}

\DeclareMathOperator{\codim}{codim}

\begin{document}%%%%%%%%%%%%%%%%%%%%%%%%%%%%%%%%%%%%%%%%%%%%%%%%%%%%%%%%%%%%%%%%%%%%%

\title{Deformations of free and linear free divisors}

\begin{abstract}
We investigate deformations of free and linear free divisors. We introduce a complex similar to the de Rham complex whose cohomology calculates deformation spaces. This cohomology turns out to be zero for all reductive linear free divisors and to be constructible for Koszul and weighted homogeneous free divisors.
\end{abstract}

\author{Michele Torielli} 

\address{\begin{flushleft}\rm {\texttt{M.Torielli@warwick.ac.uk \; (http://www.maths.warwick.ac.uk/$\sim$marhan), } } \\ Mathematics Institute, 
 University of Warwick \\ Coventry CV4 7AL, UK \end{flushleft}}
\date{\today}
\maketitle

\tableofcontents

\section{Introduction}
In this article, we develop some ideas of a deformation theory for germs of free and linear free divisors. Free divisors were introduced by K. Saito in \cite{saito} and linear free divisors by R.-O. Buchweitz and D. Mond in \cite{quiver}. Free divisors are quite fundamental in singularity theory, for example, the discriminants of the versal unfoldings of isolated hypersurfaces and complete intersection singularities are always free divisors. 

A reduced divisor $D=V(f)\subset  \C^n$ is \emph{free} if the sheaf $\Der(-\log D):=\{\delta\in\Der_{\C^n}~|~\delta(f)\in(f)\mathcal{O}_{\C^n}\}$  of logarithmic vector fields is a locally free $\mathcal{O}_{\C^n}$-module, where $\Der_{\C^n}$ denote the space of vector fields on $\C^n$. It is \emph{linear} if, furthermore, $\Der(-\log D)$ is globally generated by a basis consisting of vector fields all of whose coefficients, with respect to the standard basis $\partial/\partial x_1,\dots,\partial/\partial x_n$ of the space $\Der_{\C^n}$, are linear functions. The simplest example is the normal crossing divisor, but the main source of examples, motivating Saito's definition, has been deformation theory, where discriminants and bifurcation sets are frequently free divisors.

These objects have been studied for the past 30 years but there is still a lot to learn and discover about them. One interesting fact is that there are no examples of linear free divisors in non-trivial family. One possible approach is to deform this object in such way that each fiber of the deformation is a (linear) free divisor and that the singular locus is deformed flatly. However, not much is known on the behavior of (linear) free divisors under these kind of deformations. 

The aim of this article is to describe the spaces of infinitesimal deformations and obstructions of a germ of a (linear) free divisor and to perform calculations for some concrete examples. It turns out that the property of being a free divisor for a hypersurface $D$ has a strong influence on its deformations, in fact all free divisors $D\subset\C^n$, with $n\ge3$, are non-isolated singularities and so their space of first order infinitesimal deformations is infinite dimensional, but in what follows we will show examples of free divisors which have a finite dimensional versal deformation space as free divisors.%the free divisor deformations.

We now give an overview on the paper. The first part recalls the notions of free and linear free divisors, and describes some of their properties. In the second, we define the notion of (linearly) admissible deformations for a germ of a (linear) free divisor and we introduce a complex similar to the de Rham complex whose cohomology calculates deformations spaces. In this section we also prove our main result:
\begin{introthm} All germs of reductive linear free divisors are formally rigid.
\end{introthm} 
This is equivalent saying that for a germ of a reductive linear free divisor, there are no non-trivial families, at least on the level of formal power series.

Then, we analyse the weighted homogeneous case and we prove our second result:

\begin{introthm} If $(D,0)\subset(\C^n,0)$ is a germ of a weighted homogeneous free divisor, then it has a hull, i.e., it has a formally versal deformation.
\end{introthm}

In the last part, we describe some properties of this cohomology and we prove our third result:
\begin{introthm} If $(D,0)\subset(\C^n,0)$ is a germ of a Koszul free divisor such that we can put a logarithmic connection on $\Der_{\C^n}$ and $\Der(-\log D)$, then it has a hull.
\end{introthm}

This theory owes a lot  to the theory of deformations of Lagrangian singularities as developed in \cite{seven}, \cite{sevendipl} and \cite{deflagrangian}.

The material in this article is part of the author Ph.D. thesis \cite{Mythesis}.

\paragraph{\textbf{Acknowledgements}}
The author is grateful to Michel Granger, David Mond, Luis Narv{\'a}ez Macarro, Brian Pike, Miles Reid, Christian Sevenheck and Duco van Straten for helpful discussions on the subject of this article. We thank the anonymous referee of the Annales de l'Institut Fourier for a careful reading of our draft versions and a number of very helpful remarks.

\section{Basic notions}

Fix coordinates $x_1,\dots,x_n$  on $\C^n$ .

\begin{Definition} A reduced divisor $D=V(f)\subset \C^n$ is called \emph{free} if the sheaf $$\Der(- \log D):=\{\delta\in\Der_{\C^n}~|~\delta(f)\in(f)\mathcal{O}_{\C^n}\}$$ of logarithmic vector fields is a locally free $\mathcal{O}_{\C^n}$-module.
\end{Definition}
\begin{Definition} Let $D=V(f)\subset \C^n$ be a reduced divisor. Then for $q=0,\dots,n$, we define the sheaf $$\Omega^q(\log D):=\{\omega\in\Omega_{\C^n}^q[\star D]~|~f\omega\in\Omega_{\C^n}^q, fd\omega\in\Omega_{\C^n}^{q+1}\} $$ of $q$-forms with logarithmic poles along $D$.
\end{Definition}

Note that by definition, $\Omega^{0}(\log D)=\Omega_{\C^n}^{0}$ and $\Omega^n(\log D)=\dfrac{1}{f}\Omega_{\C^n}^n$.
\begin{Lemma} \emph{(\cite{saito}, Lemma 1.6)} By the natural pairing $$\Der_{p}(-\log D)\times \Omega_{p}^{1}(\log D)\to \mathcal{O}_{\C^n,p}~ \text{ defined by }~ (\delta,\omega)\mapsto \delta \cdot \omega,$$ each module is the $\mathcal{O}_{\C^n,p}$-dual of the other.
\end{Lemma}
\begin{Corollary}\label{planecurvefree} $\Omega_{p}^1(\log D)$ and $\Der_p(-\log D)$ are reflexive $\mathcal{O}_{\C^n,p}$-modules. In particular, when $n=2$, then $\Omega_p^1(\log D)$ and $\Der_p(-\log D)$ are free $\mathcal{O}_{\C^2,p}$-modules. 
\end{Corollary}
\begin{Definition} A free divisor $D$ is \emph{linear} if there is a basis for $\Gamma(\C^n , \Der(-\log D))$ as $\C[\C^n]$-module consisting of vector fields all of whose coefficients, with respect to the standard basis $\partial / \partial x_1, \dots, \partial / \partial x_n$ of the space $\Der_{\mathbb{C}^{n}}$, are linear functions, i.e. they are all homogeneous polynomials of degree $1$.  \end{Definition}
 \begin{Remark} With respect to the standard grading of $\Der_{\mathbb{C}^{n}}$, i.e., $\deg x_i = 1$ and $\deg \partial / \partial x_i = -1$ for every $i = 1, \dots, n$, such vector fields have weight zero.
 \end{Remark}
 \begin{Definition} We denote by $\Der(-\log D)_0$ the finite dimensional Lie subalgebra of $ \Der(-\log D)$ consisting of the weight zero logarithmic vector fields.
 \end{Definition}
 
There is a nice criterion to understand easily if a divisor is free or not:
\begin{Proposition}\label{saitocrit} \emph{(SAITO'S CRITERION) (\cite{saito}, Theorem 1.8)} i) The hypersurface $D \subset  \mathbb{C}^{n}$ is a free divisor in the neighbourhood of a point $p$ if and only if $\bigwedge^n \Omega_{p}^1(\log D)=\Omega_{p}^n(\log D)$, i.e. if there exist n elements $\omega_1,\dots,\omega_n \in \Omega_{p}^1(\log D)$ such that $$\omega_1\wedge \cdots \wedge \omega_n=\alpha \dfrac{dx_1\wedge \cdots \wedge dx_n}{f}$$ where $\alpha$ is a unit. Then the set of forms $\{\omega_1,\dots,\omega_n \}$ form a basis for $\Omega_{p}^1(\log D)$. Moreover, we have $$\Omega_{p}^q(\log D)=\bigoplus_{i_1<\dots<i_q}\mathcal{O}_{\C^n,p}\omega_{i_1}\wedge \cdots \wedge \omega_{i_q}$$ for $q=1,\dots,n$.

ii) The hypersurface $D \subset  \mathbb{C}^{n}$ is a free divisor in the neighbourhood of a point $p$ if and only if there exist germs of vector fields $\chi_1, \dots, \chi_n \in \Der_p(-\log D)$ such that the determinant of the matrix of coefficients $[\chi_1, \dots, \chi_n]$, with respect to some, or any, $\mathcal{O}_{ \mathbb{C}^{n}, p}$-basis of $\Der_{\mathbb{C}^{n},p}$ is a reduced equation for $D$ at $p$ i.e. it is a unit multiple of $f_p$. In this case, $\chi_1, \dots, \chi_n$ form a basis for $\Der_p(-\log D)$.
\end{Proposition}
\begin{Definition}
In the notation of Proposition \ref{saitocrit}, the matrix $[\chi_1, \dots, \chi_n]$ is called a \emph{Saito matrix}.
\end{Definition}
\begin{Lemma}\label{saitolemma}\emph{(\cite{saito}, Lemma 1.9)} Let $\delta_i=\sum_{j=1}^n a_i^j(x)\partial / \partial x_j$, $i=1, \dots, n$, be a system of holomorphic vector fields at p such that
\begin{enumerate}
	\item $[\delta_i,\delta_j]\in \sum_{k=1}^n \mathcal{O}_{ \mathbb{C}^{n}, p} \delta_k$ for $i,j=1, \dots, n$;
	\item $\det(a_i^j)=f$ defines a reduced hypersurface $D$.
\end{enumerate}
Then for $D=\{ f(x)=0\}$, $\delta_1, \dots, \delta_n$ belong to $\Der_p(-\log D)$, and hence $\{ \delta_1, \dots, \delta_n\}$ is a free basis of $\Der_p(-\log D)$.
\end{Lemma}

There is also an algebraic version of Saito's criterion that does not refer to vector fields directly but characterizes the Taylor series of the function $f$ defining a free divisor: 
\begin{Proposition}\label{otherfreediv}\emph{(\cite{quiver}, Proposition 1.3)} A formal power series $f \in R=\C[[x_1, \dots, x_n]]$ defines a free divisor, if it is reduced, i.e. squarefree, and there is an $n\times n$ matrix $A$ with entries from R such that $$\emph{det } A=f ~ \text{  and  } ~ (\nabla f)A \equiv (0, \dots , 0)\text{\emph{ mod }}f,$$ where $\nabla f=(\partial f/ \partial x_1, \dots , \partial f/ \partial x_n)$ is the gradient of $f$, and the last condition just expresses that each entry of the vector $(\nabla f)A$ is divisible by $f$ in R. The columns of A can then be viewed as the coefficients of a basis, with respect to the derivations $\partial /\partial x_i$, of the logarithmic vector fields along the divisor $f=0$.
\end{Proposition} 
\begin{Example}\label{ncrdiv}
The normal crossing divisor $D$ = $\{x_1 \cdots x_n  = 0\} \subset \mathbb{C}^{n}$ is a linear free divisor; $\Der(-\log D)$ has basis $x_1\partial / \partial x_1, \dots, x_n\partial / \partial x_n$. Up to isomorphism it is the only example among hyperplane arrangements, see  \cite{orlterao}, Chapter 4.
%because it is the only hyperplane arrangement defined by a homogeneous reduced equation of degree $n$.
\end{Example}
\begin{Remark} Let $D \subset \C^n$ be a divisor defined by a homogeneous polynomial $f\in \C[x_1, \dots x_n]$ of degree n. Then for each $\delta\in \Der(-\log D)_0$, there is a $n\times n$ matrix $A$ with entries in $\C$, such that $\delta=xA\partial^t$, where $\partial^t$ is the column vector $(\partial / \partial x_1, \dots, \partial / \partial x_n)^t$.
\end{Remark}
\begin{Remark} Let $D\subset \C^n$ be a free divisor. $D$ is a linear if and only if $\Der(-\log D)=\mathcal{O}_{\C^n}\cdot\Der(-\log D)_0$. 
\end{Remark}
\begin{Definition} Let $D=V(f) \subset \C^n$ be a linear free divisor. Define the subgroup $$G_D:=\{A \in GL_{n} (\C)~|~A(D)=D\} = \{A \in GL_n(\C)~|~f\circ A \in \C \cdot f \} $$ with identity component $G^\circ_D$ and Lie algebra $\mathfrak{g}_D$.
\end{Definition}
\begin{Lemma}\label{gdlemma}\emph{(\cite{grmondsch}, Lemma 2.1)} $G^\circ_D$ is an algebraic subgroup of  $\GL_n(\C)$ and $\mathfrak{g}_D=\{A ~| ~xA^t\partial^t \in \Der(-\log D)_0 \}$.
\end{Lemma}
\begin{Definition} Let $D\subset \C^n$ be a linear free divisor. We call $D$ \emph{reductive} if $\mathfrak{g}_D$ is a reductive Lie algebra. 
\end{Definition}
From $\S$2 of \cite{ignacmond}, we can deduce the following: 
\begin{Lemma}\label{reductrace} Let $D=V(f)\subset\C^n$ be a reductive linear free divisor. Then $\Aut(f) \subset SL_n(\C)$. This means that if $\chi \in \Ann(D):=\{\delta \in \Der(-\log D)~|~\delta(f)=0 \}$ then $\trace(\delta)=0$.
\end{Lemma}
\begin{Example}
\begin{enumerate}
	\item[i)] The normal crossing divisor of Example \ref{ncrdiv} is a reductive linear free divisor because $\mathfrak{g}_D=\C^n$.
	\item[ii)] Consider the divisor $D=V((y^2+xz)z) \subset \C^3$. This is a linear free divisor because we can take the matrix 
	\begin{equation*} A=
\begin{bmatrix}
		x&4x&-2y\\
		y&y&z\\
		z&-2z&0\\
\end{bmatrix}
\end{equation*}
as its Saito matrix. Moreover, if we consider $\sigma$ the second column of $A$, i.e. $\sigma=4x\partial/\partial x+ y \partial/\partial y -2z \partial/\partial z$, we have that $\sigma \in \Ann(D)$ and $\trace(\sigma)=3$ and hence by Lemma \ref{reductrace}, $D$ is a non-reductive linear free divisor.
\end{enumerate}
\end{Example}
\begin{Lemma}\label{reductgroup}\emph{(\cite{grmondsch}, Lemma 3.6, (4))}  Let $D\subset\C^n$ be a linear free divisor. If $\mathfrak{g}_D$ is reductive then $G^\circ_D$ is reductive as algebraic group.
\end{Lemma}
\begin{Definition} Let $S$ be a complex space. Then $\Der_{\C^n \times S/S}$ is the set of vector fields on $\C^n \times S$ without components in the $S$ direction. It is a submodule of $\Der_{\C^n\times S}$. 
\end{Definition}
\begin{Definition} Let $S$ be a complex space and let $D \subset \C^n \times S$ be a divisor. Then $\Der(-\log D/S):=\{\delta \in \Der(-\log D)~|~\delta \in \Der_{\C^n \times S/S} \}=\Der(-\log D)\cap\Der_{\C^n\times S/S}$.
\end{Definition}
\begin{Remark} $\Der_{\C^n \times S/S}$ and $\Der(-\log D/S)$ are both coherent sheaves of $\mathcal{O}_{\C^n\times S}$-modules.
\end{Remark}

\section{Deformation theory for free divisors}

The aim of this section is to introduce the notion of (linearly) admissible deformation for germs of (linear) free divisors and then study infinitesimal ones in order to prove that reductive linear free divisors are formally rigid.

%In the remainder of the paper, we will talk about germs of complex spaces but to make the writing easier, we will call them just complex spaces and we will denote them by $D$ in place of $(D,0)$.

\subsection{Admissible and linearly admissible deformations}

\begin{Definition}\label{defdeform} Let $(D,0)=(V(f),0)\subset  (\C^n,0)$ be a germ of a free divisor and let $(S,s)$ be a complex space germ. An \emph{admissible deformation} of $(D,0)$ over $(S,s)$ consists of a flat morphism $\phi\colon(X,x)\to (S,s)$ of complex space germs, where $(X,x)\subset(\C^n\times S,(0,s))$, together with an isomorphism from $(D,0)$ to the central fibre of $\phi, (D,0) \to (X_s,x) :=(\phi^{-1}(s),x)$, such that
\begin{equation}\label{relderadmisdef}
\Der(-\log X/S)/\mathfrak{m}_{S,s}\Der(-\log X/S)=\Der(-\log D)
\end{equation}
where $\mathfrak{m}_{S,s}$ is the maximal ideal of $\mathcal{O}_{S,s}$.

Moreover, if $(D,0)$ is linear, we define a \emph{linearly admissible deformation} of $(D,0)$ over $(S,s)$ as an admissible deformation of $(D,0)$ over $(S,s)$ such that there exists a basis of  $\Der(-\log X/S)$ as $\mathcal{O}_{\C^n \times S,(0,s)}$-module consisting of vector fields all of whose coefficients are linear in $x_1, \dots, x_n$.
\end{Definition}
\begin{Definition} In Definition \ref{defdeform}, $(X,x)$ is called the \emph{total space}, $(S,s)$ the \emph{base space} and $(X_s,x) \cong (D,0)$ the \emph{special fibre} of the (linearly) admissible deformation.
\end{Definition}
We can write a (linearly) admissible deformation as a commutative diagram
\begin{equation}
\xymatrix{(D,0) \ar[d]  \ar@{^{(}->}[r]^i &(X,x) \ar[d]^\phi \\
 \{\ast \} \ar@{^{(}->}[r] &(S,s)} \label{defdiag} \end{equation}
 where $i$ is a closed embedding mapping $(D,0)$ isomorphically onto $(X_s,x)$. We will denote a (linearly) admissible deformation by
 $$(i,\phi)\colon \xymatrix{(D,0)\ar@{^{(}->}[r]^i &(X,x)\ar[r]^\phi &(S,s)}.$$
 \begin{Definition} Given two (linearly) admissible deformations $(i,\phi)\colon D\hookrightarrow X\to S$ and $(j,\psi)\colon D\hookrightarrow Y\to T$, of $D$ over $S$ and $T$ respectively. A \emph{morphism of (linearly) admissible deformations} from $(i,\phi)$ to $(j,\psi)$ is a morphism of the diagram \eqref{defdiag} being the identity on $D \to \{\ast \}$. Hence, it consists of two morphisms $(\tau,\sigma)$ such that the following diagram commutes
$$\xymatrix{&D \ar@{^{(}->}[ld]_i \ar@{^{(}->}[rd]^j\\
X \ar[d]_\phi \ar[rr]^\tau &&Y \ar[d]^\psi\\
S \ar[rr]^\sigma &&T}$$
\end{Definition}
 \begin{Definition} Two (linearly) admissible deformations over the same base space $S$ are \emph{isomorphic} if there exists a morphism $(\tau,\sigma)$ with $\tau$ an isomorphism and $\sigma$ the identity map.
 \end{Definition}
% \begin{Remark}\label{defarehyper} In Definition \ref{defdeform}, the requirement that $\phi$ be a flat morphism implies that $(X,x)$ is a germ of a hypersurface in $(\C^n \times S,(0,s))$. \end{Remark}
 We denote by $\mathbf{Art}$ the category of local Artin rings with residue field $k$ and by $\mathbf{Set}$ the category of pointed sets with distinguished element $\ast$.
 %\begin{Definition} Let $D\subset \C^n$ be a (linear) free divisor and $A\in \mathbf{Art}$. A \emph{formal (linearly) admissible deformation} of $D$ over $A$ is a (linearly) admissible deformation of $D$ over $\Spec A$.
%\end{Definition}
\begin{Definition} Let $(D,0)\subset (\C^n,0)$ be a germ of a free divisor. Define the functor $\mathbf{FD}_D\colon \mathbf{Art} \to \mathbf{Set}$ by setting\begin{equation*}
\mathbf{FD}_D(A):=\left\{
\begin{array}{ccc}
\text{Isomorphism classes of admissible} \\
\text{deformations of } (D,0) \text{ over } \Spec A 
\end{array} \right \}.
\end{equation*}
If $(D,0)\subset (\C^n,0)$ is a germ of a linear free divisor, we define similarly the functor $\mathbf{LFD}_D\colon \mathbf{Art} \to \mathbf{Set}$ by setting 
\begin{equation*}
\mathbf{LFD}_D(A):=\left\{
\begin{array}{ccc}
\text{Isomorphism classes of linearly } \\
\text{admissible deformations of } (D,0) \text{ over } \Spec A 
\end{array} \right \}.
\end{equation*}
\end{Definition}
\begin{Theorem}\label{isdeffunct} Let $(D,0)\subset (\C^n,0)$ be a germ of a free divisor. Then the functor $\mathbf{FD}_D$ satisfies Schlessinger's conditions (H1) and (H2) from \cite{schless}. Moreover, if $(D,0)$ is linear, then also the functor $\mathbf{LFD}_D$ satisfies conditions (H1) and (H2). 
\end{Theorem}
\begin{proof} Let $A'\to A$ and $A'' \to A$ be maps in $\mathbf{Art}$ such that the latter is a small extension, see Definition 1.2 from \cite{schless}. Consider now $X\in \mathbf{FD}_D(A), X'\in \mathbf{FD}_D(A')$ and $X''\in \mathbf{FD}_D(A'')$. Define $Y:=(D, \mathcal{O}_{X'}\times_{\mathcal{O}_X} \mathcal{O}_{X''})$, by \cite{schless}, Lemma 3.4, it is flat over $A'\times_A A''$ and it is an element of $\mathbf{FD}_D(A'\times_A A'')$. Hence the map $\tau_{A', A'', A}$ of $(H1)$ is surjective.

We want to show now that $\tau_{A', A'', A}$ is a bijection in the case $A''=k[\epsilon]$ and $A=k$. Let $W\in \mathbf{FD}_D(A'\times_A A'')$ restrict to $X'$ and $X''$, then we can choose immersions $q' : X' \hookrightarrow W$ and $q'' \colon X'' \hookrightarrow W$. Since these maps are all compatible with the immersions from $D$, they agree with the chosen maps $u' : X \hookrightarrow X'$ and $u'' : X \hookrightarrow X''$, since in this case $X=D$. Now by the universal property of fibered product of rings, there is a map $Y\to W$ compatible with the above maps. Since $Y$ and $W$ are both flat over $A'\times_A A''$, and the map becomes an isomorphism when restricted to $D$, we find that, by \cite{harts}, Exercise 4.2, $Y$ is isomorphic to $W$ and hence they are equal as elements of $\mathbf{FD}_D(A'\times_A A'')$.

The previous proof works similarly also for the functor $\mathbf{LFD}_D$.
\end{proof}
\begin{Proposition}\label{flatdefsingloc} Let $(D,0)\subset(\C^n,0)$ be a germ of a free divisor. Then in any admissible deformation the singular locus of $(D,0)$ is deformed in a flat way.
 \end{Proposition}
 \begin{proof} Let $f\in\mathcal{O}_{\C^n,0}$ be a defining equation for $(D,0)$ and let $\phi\colon(X,x)\to (S,s)$ be a admissible deformation of $(D,0)$. Any element of $\Der(-\log D)$ can be seen as a relation among $f,\partial f/\partial x_1,\dots,\partial f/\partial x_n$ and similarly, any element of $\Der(-\log X/S)$ can be seen as a relation among $F,\partial F /\partial x_1,\dots,\partial F/\partial x_n$, where $F\in\mathcal{O}_{\C^n\times S,(0,s)}$ is a defining equation for $(X,x)$. The requirement \eqref{relderadmisdef} of Definition \ref{defdeform} implies then that any relation among $f,\partial f/\partial x_1,\dots,\partial f/\partial x_n$ lifts to a relation among $F,\partial F /\partial x_1,\dots,\partial F/\partial x_n$ and this is equivalent to the deformation of the singular locus of $(D,0)$ being flat. See \cite{intdef}, Chapter I, Proposition 1.91. 
 \end{proof}
 % \begin{Proposition}\label{flatdefsingloc} Let $f\in\mathcal{O}_{\C^n}$ be a defining equation for $D$. Any element of $\Der(-\log D)$ can be seen as a relation among $f,\partial f/\partial x_1,\dots,\partial f/\partial x_n$ and similarly, any element of $\Der(-\log X/S)$ can be seen as a relation among $F,\partial F /\partial x_1,\dots,\partial F/\partial x_n$, where $F$ is a defining equation for $(X,x)$. The requirement $\Der(-\log X/S)\otimes \mathcal{O}_{\C^n}=\Der(-\log D)$ of definition \ref{defdeform} implies then that any relation among $f,\partial f/\partial x_1,\dots,\partial f/\partial x_n$ lifts to a relation among $F,\partial F /\partial x_1,\dots,\partial F/\partial x_n$ and this is equivalent to deform the singular locus of $D$ in a flat way.
 %\end{Proposition}
  \begin{Proposition}\label{freerelderlog} In the situation of Definition \ref{defdeform}, requirement \eqref{relderadmisdef} implies that $\Der(-\log X/S)$ is a locally free $\mathcal{O}_{\C^n \times S,(0,s)}$-module of rank n.
 \end{Proposition}
 \begin{proof} By Proposition \ref{flatdefsingloc}, the singular locus of $(D,0)$ is deformed flatly and so $\mathcal{O}_{\C^n\times S,(0,s)}/I$ is a flat $\mathcal{O}_{S,s}$-module and represents a deformation of $\mathcal{O}_{\C^n,0}/I_0$, where $I=(F,\partial F /\partial x_1,\dots,\partial F/\partial x_n)$ and $I_0=(f,\partial f/\partial x_1,\dots,\partial f/\partial x_n)$. Hence, a free resolution of $\mathcal{O}_{\C^n,0}/I_0$ lifts to a free resolution of $\mathcal{O}_{\C^n\times S,(0,s)}/I$. Because $(D,0)$ is free, then a free resolution of $\mathcal{O}_{\C^n\times S,(0,s)}/I$ looks like
 $$\xymatrix{0\ar[r]& \mathcal{O}_{\C^n\times S,(0,s)}^n\ar[r]&\mathcal{O}_{\C^n\times S,(0,s)}^{n+1}\ar[rrr]^-{(F,\partial F /\partial x_1,\dots,\partial F/\partial x_n)}&&&\mathcal{O}_{\C^n\times S,(0,s)}\ar[r]&\mathcal{O}_{\C^n\times S,(0,s)}/I\ar[r] &0} $$
 But as explained in Proposition \ref{flatdefsingloc}, we can identify $\Der(-\log X/S)$ with the syzygy module of $(F,\partial F /\partial x_1,\dots,\partial F/\partial x_n)$, and hence, it is locally free of rank $n$.
 \end{proof}
 \begin{Remark} In our theory, we require more than only that each fiber is a free divisor. In fact, let $(D,0)\subset(\C^n,0)$ be a singular free divisor with a quasi-homogeneous equation $f$. Then we can consider $(X,0)=(V(f-t),0)\subset (\C^n\times \C,0)$ and $\phi$ the projection on $(\C,0)$. In this case each fiber is a free divisor but this is not an admissible deformation of $(D,0)$.  
 \end{Remark}
 \begin{proof} Because $f$ is quasi-homogeneous, we can take $\chi,\sigma_1,\dots,\sigma_{n-1}$ as a basis of $\Der(-\log D)$, where $\chi=\sum_{i=1}^n\alpha_ix_i\partial /\partial x_i$ with $\alpha_1,\dots,\alpha_n\in \C$ is the Euler vector field and $\sigma_1,\dots,\sigma_{n-1}$ annihilate $f$. Hence $\chi(f)=\sum_{i=1}^n\alpha_ix_i\partial f/\partial x_i= f$. Notice that because $(X,0)$ is non-singular, it is a free divisor in $(\C^n\times\C,0)$ and so we can take as Saito matrix for $(X,0)$, the matrix
 \begin{equation*} A=
\begin{bmatrix}
		1&0&\cdots&0&0\\
		0&1&\cdots&0&0\\
		0&0&\cdots&0&0\\
		\vdots&\vdots& &\vdots& \vdots\\
		0&0&\cdots&1&0\\
		\partial f/\partial x_1&\partial f/\partial x_2&\cdots&\partial f/\partial x_n&f-t\\
\end{bmatrix}
\end{equation*}
Let $\lambda_i$ be the vector field represented by the $i$-th column of $A$. Consider now the vector fields $\sigma_i^*=\sigma_i$ seen as a vector field in $\C^n\times\C$ and $\tau_i=t\lambda_i+\partial f/\partial x_i\lambda_{n+1}-\partial f/\partial x_i\sum_{j=1}^n\alpha_jx_j\lambda_j$. Clearly, $\sigma_i^*(f-t)=\sigma_i(f)=0$ and so $\sigma_i^*\in\Der(-\log X/\C)$. Similarly, $\tau_i\in\Der(-\log X/\C)$ because $\tau_i\in\Der(-\log X)$ and its coefficient of $\partial/\partial t$ is equal to $t\partial f/\partial x_i +\partial f/\partial x_i(f-t)-\partial f/\partial x_i\sum_{j=1}^n\alpha_jx_j\partial f/\partial x_j=\partial f/\partial x_i(f-\chi(f))=0$. This implies that we have an inclusion $\langle\sigma_1^*,\dots,\sigma_{n-1}^*,\tau_1,\dots,\tau_n \rangle\subset\Der(-\log X/\C)$. However, because $\sigma_1,\dots,\sigma_{n-1}$ are the generators of $\Ann(f):=\{\delta\in\Der(-\log D)~|~\delta(f)=0\}$, then any element of $\Der(-\log X/\C)$ that is a linear combination of $\lambda_1,\dots,\lambda_{n}$ is a linear combinations of $\sigma_1^*,\dots,\sigma_{n-1}^*$. Consider now an element of $\Der(-\log X/\C)$ that can be written as a linear combination of the $\lambda_i$ involving $\lambda_{n+1}$. Because it is independent of $\partial/\partial t$, then the coefficient of $\lambda_{n+1}$ is forced to be in the Jacobian ideal of $f$. Because $t$ appear only in $\lambda_{n+1}$, this implies that, modulo the $\sigma^*_i$, it is a linear combination of $\tau_1,\dots,\tau_n$. Hence $\sigma_1^*,\dots,\sigma_{n-1}^*,\tau_1,\dots,\tau_n$ generate $\Der(-\log X/\C)$. 

Because $f$ is singular, $\partial f/\partial x_i\in(x_1,\dots,x_n)$ for all $i=1,\dots,n$ and so each $\tau_i$ has weight bigger than zero, i.e. $\deg(\partial f/\partial x_i\alpha_jx_j)-\deg(x_j)>0$. This tells us that the Euler vector field $\chi\notin \Der(-\log X/\C)/\mathfrak{m}_{\C,0}\Der(-\log X/\C)$ because $\chi$ has weight zero and is not a linear combination of $\sigma_1,\dots,\sigma_{n-1}$.
 \end{proof}
 \begin{Remark} If $f$ is non-singular, then the deformation defined in the previous Remark is an admissible deformation.
 \end{Remark}
 \begin{proof} We can suppose $f=x_1$ and we can take as Saito matrix
 \begin{equation*} 
\begin{bmatrix}
		x_1&0&0&\cdots&0\\
		0&1&0&\cdots&0\\
		0&0&1&\cdots&0\\
		\vdots&\vdots&\vdots& &\vdots\\
		0&0&0&\cdots&1\\
\end{bmatrix}
\end{equation*}
By a similar argument as the proof of the previous Remark, $\Der(-\log X/\C)$ is generated by the columns of the matrix
\begin{equation*} 
\begin{bmatrix}
		x_1-t&0&0&\cdots&0\\
		0&1&0&\cdots&0\\
		0&0&1&\cdots&0\\
		\vdots&\vdots&\vdots& &\vdots\\
		0&0&0&\cdots&1\\
		0&0&0&\cdots&0\\
\end{bmatrix}
\end{equation*}
and hence the requirement \eqref{relderadmisdef} of the Definition \ref{defdeform} is fulfilled.
 \end{proof}
 \begin{Remark}\label{trivialadmisdefgerm} Let $(i,\phi)\colon(D,0)\hookrightarrow (X,x)\to (S,s)$ be a (linearly) admissible deformation. Then it is a trivial  (linearly) admissible deformation if and only if it is trivial as deformation of $(D,0)$ as complex space germ.
 \end{Remark}
 \begin{Definition} The complex space $T_\epsilon$ consists of one point with local ring $\C[\epsilon]=\C + \epsilon \cdot \C, \epsilon^2=0$, that is, $\C[\epsilon]=\C[t]/(t^2)$, where t is an indeterminate. Thus $T_\epsilon = \Spec(\C[t]/(t^2))$.
 \end{Definition}
 \begin{Definition} An \emph{infinitesimal (linearly) admissible deformation} of a germ of a (linear) free divisor $(D,0) \subset (\C^n,0)$ is a (linearly) admissible deformation of $(D,0)$ over $T_\epsilon$.
 \end{Definition}
 \begin{Definition} Let $(D,0)\subset  (\C^n,0)$ be a germ of a free divisor. Then $\mathcal{FT}^1(D):=\mathbf{FD}_D(\C[t]/(t^2))$. Similarly if $(D,0)$ is linear, then $\mathcal{LFT}^1(D):=\mathbf{LFD}_D(\C[t]/(t^2))$.
\end{Definition}
% \begin{Definition} Let $f\in \C[[x_1, \dots, x_n]]$ such that $f(0)=0$, then an \emph{unfolding} of $f$ is a power series $F\in \C[[x_1, \dots, x_n, t_1, \dots, t_r]]$ with $F(x,0)=f(x)$, that is: $$F(x,t)=f(x)+\sum_{\substack{i \in \mathbb{N}^r \\ i\ne(0,\dots,0)}} g_i(x)t^i$$ for some $g_i\in \C[[x_1, \dots, x_n]]$.
 %\end{Definition}
 \begin{Proposition}\label{cartdiagram} \emph{(\cite{intdef}, Chapter II, Proposition 1.5)} Consider a commutative diagram of complex space germs
 \begin{equation*}
\xymatrix{(X_0,x) \ar[d]_{f_0}  \ar@{^{(}->}[r] &(X,x) \ar[d]^f \\
 (S_0,s) \ar@{^{(}->}[r] &(S,s)} \end{equation*}
 where the horizontal maps are closed embeddings. Assume that $f_0$ factors as
  \begin{equation*} \xymatrix{(X_0,x) \ar@{^{(}->}[r]^-{i_0} &(\C^n,0)\times (S_0,s) \ar[r]^-{p_0} &(S_0,s)} \end{equation*}
  with $i_0$ a closed embedding and $p_0$ the second projection. Then there exists a commutative diagram 
  \begin{equation*}
\xymatrix{(X_0,x) \ar@/_4pc/[dd]_{f_0} \ar@{^{(}->}[d]_{i_0}  \ar@{^{(}->}[r] &(X,x)\ar@/^4pc/[dd]^f \ar@{^{(}->}[d]^i \\
 (\C^n,0)\times (S_0,s) \ar[d]_{p_0} \ar@{^{(}->}[r] &(\C^n,0)\times (S,s) \ar[d]^p \\
 (S_0,s)\ar@{^{(}->}[r] &(S,s)} \end{equation*}
 with $i$ a closed embedding and $p$ the second projection. That is, the embedding of $f_0$ over $(S_0,s)$ extends to an embedding of $f$ over $(S,s)$.
 \end{Proposition}
 \begin{Corollary} Any (linearly) admissible deformation of $(D,0)=(V(f),0) \subset (\C^n,0)$ over a complex space germ $(S,s)$ is of the form $(X,(0,s))=(V(F),(0,s))\subset (\C^n \times S,(0,s))$, for some unfolding $F$ of $f$ with $\phi$ just the projection on $(S,s)$.
 \end{Corollary}
  \begin{Remark}\label{forminfdef} Any infinitesimal (linearly) admissible deformation of $(D,0)=(V(f),0) \subset (\C^n,0)$ is of the form $(X,0)=(V(f+ \epsilon \cdot f'),0) \subset (\C^n \times T_\epsilon,0)$, for some $f' \in  \mathcal{O}_{\C^n,0}$, where $\phi$ is just the projection on $T_\epsilon$.
 \end{Remark}
 By Remark \ref{trivialadmisdefgerm} and Chapter II, 1.4 from \cite{intdef}, we have the following:
 \begin{Remark}\label{trivial} An infinitesimal (linearly) admissible deformation $(X,0)=(V(f+ \epsilon \cdot f'),0) \to T_\epsilon$ is trivial if and only if there is an isomorphism $$\mathcal{O}_{\C^n\times T_\epsilon,0}/(f) \cong \mathcal{O}_{\C^n\times T_\epsilon,0}/(f+ \epsilon \cdot f')$$ which is the identity modulo $\epsilon$ and which is compatible with the inclusion of $\mathcal{O}_{T_\epsilon}$ in $\mathcal{O}_{\C^n\times T_\epsilon,0}$. Such an  isomorphism is induced by an automorphism $\varphi$ of $\mathcal{O}_{\C^n\times T_\epsilon,0}$, mapping $x_j \mapsto x_j + \epsilon \sigma_j(x)$ and $\epsilon \mapsto \epsilon$ such that $$(\varphi ^\ast f)=(f(x+\epsilon \cdot \sigma(x))) = (f+ \epsilon \cdot f'),$$ where $x=(x_1, \dots, x_n)$ and $\sigma=\sum_{j=1}^n\sigma_j\partial/\partial x_j$.
 \end{Remark}
 We now prove a relative Saito's Lemma in order to be able to characterise an (linearly) admissible deformation by logarithmic vector fields.
 \begin{Lemma}\label{detrelvect} Let $(S,s)$ be a complex space germ with an embedding $(S,s)\subset(\C^r,0)$ and let $t=(t_1, \dots, t_r)$ be coordinates on the ambient space $(\C^r,0)$. Let $(X,x)\subset (\C^n \times S,(0,s))$ be a (linearly) admissible deformation of a germ of a (linear) free divisor $(D,0) \subset (\C^n,0)$ and let $h_p = 0$ be a reduced equation for $(X,x)$, locally at $p=(x_0, t_0) \in (\C^n \times S,(0,s))$. Suppose $\delta'_i=\sum_{j=1}^n a_i^j(x,t)\partial / \partial x_j \in \Der_p(-\log X/S)$, $\forall ~ i=1, \dots, n$, then $\det(a_i^j)\in (h_p)\mathcal{O}_{\C^n \times S,p}$.
\end{Lemma}
\begin{proof} Suppose that $\det(a_i^j)$ does not vanish at $p$, hence it does not vanish in a small neighbourhood $U$ of $p$. This implies that $\delta'_1, \dots, \delta'_n$ are linearly independent in $U$. Consider now the fibre $X_{t_0}$. We have that $\widetilde{\delta'_i}=\sum_{j=1}^n a_i^j(x,t_0)\partial / \partial x_j \in \Der(-\log X_{t_0})$ and are linearly independent, but this implies that $X_{t_0}$ is $n$-dimensional, contradicting the fact that $(X,x)$ is a flat (linearly) admissible deformation of $(D,0)$, that is $(n-1)$-dimensional. 
\end{proof}
\begin{Proposition}\label{generalizsaitocrit1} Let $(S,s)$ be a complex space germ with an embedding $(S,s)\subset(\C^r,0)$ and let $t=(t_1, \dots, t_r)$ be coordinates on the ambient space $(\C^r,0)$. Let $(X,x)\subset (\C^n \times S,(0,s))$ be a (linearly) admissible deformation of a germ of a (linear) free divisor $(D,0) \subset (\C^n,0)$ and let $h_p = 0$ be a reduced equation for $(X,x)$, locally at $p=(x_0, t_0) \in (\C^n \times S,(0,s))$. Then there exist $\delta'_1,\dots,\delta'_n\in \Der_p(-\log X/S)$ with $\delta'_i=\sum_{j=1}^n a_i^j(x,t)\partial / \partial x_j$, such that  $\det(a_i^j)$ is a unit multiple of $h_p$.
\end{Proposition}
\begin{proof} By Proposition \ref{freerelderlog}, $\Der_p(-\log X/S)$ is a free $\mathcal{O}_{\C^n\times S,p}$-module of rank $n$. Since $\Der(-\log X/S)$ is coherent, there exists a neighbourhood $U$ of $p$ such that $\Der(-\log X/S)|_U$ is free. Let $\delta'_1,\dots,\delta'_n$ be a basis of $\Der(-\log X/S)|_U$ with $\delta'_i=\sum_{j=1}^n a_i^j(x,t)\partial / \partial x_j$. By Lemma \ref{detrelvect}, $\det(a_i^j)=g h_p$, where $g$ is a holomorphic function on $U$. Since $\partial/\partial x_1,\dots,\partial/\partial x_n$ is a basis for $p\in U\setminus X$, then $g$ does not vanish on $U\setminus X$. At a smooth point $p\in X$, we can suppose $X=V(x_1)$ and hence, we may choose as a basis of $\Der(-\log X/S)$ on $X_{reg}\cap U$ the vector fields $x_1\partial/\partial x_1,\dots,\partial/\partial x_n$. Thus $g$ does not vanish anywhere on $U\setminus(U\cap X_{sing})$, but because $\codim_U(U\cap X_{sing})>1$, then $g$ does not vanish anywhere on $U$ and so it is a unit.
\end{proof}
%By \cite{lundell}, $\delta_1, \dots, \delta_n$ span a $n$ dimensional completely integrable distribution on $U$, thus  $Der(-log~X/S)$ is generated by $\partial / \partial x_1, \dots, \partial / \partial x_n$ in $U$, but this contradict the fact that $j^\ast(Der(-log~X/S))=Der(-log~D)$, where $j:\C^n\times S \to \C^n$.\end{proof}    
\begin{Lemma}\label{molteachcolumntrace} Let $R$ be a commutative ring, $A$ and $B$ be two $n\times n$ matrices and $a_1,\dots, a_n$ be the columns of $A$. Then
$$\sum_{i=1}^n\det[a_1,\dots,a_{i-1},Ba_i,a_{i+1},\dots,a_n]=\trace(B)\det(A).$$
\end{Lemma}
\begin{proof} It is know that if we consider a $n\times n$ matrix $C$ with columns $c_1,\dots,c_n$, then
$$d_A\det(C)=\sum_{i=1}^n\det[a_1,\dots,a_{i-1},c_i,a_{i+1},\dots,a_n],$$
where $d$ is the tangent map.
Then we have the following equalities
$$\sum_{i=1}^n\det[a_1,\dots,a_{i-1},Ba_i,a_{i+1},\dots,a_n]=d_A\det(BA)=\frac{d}{dt}(\det(A+tBA))\lvert_{t=0}=$$
$$\det(A)\frac{d}{dt}(\det(I+tB))\lvert_{t=0}=\det(A)d_I\det(B)=\det(A)\trace(B).$$
\end{proof}

\begin{Lemma}\label{relatversaitlem} Let $(S,s)$ be a complex space germ  with an embedding $(S,s)\subset(\C^r,0)$ and let $t=(t_1, \dots, t_r)$ be coordinates on the ambient space $(\C^r,0)$. Consider $(D,0)=(V(f),0) \subset (\C^n,0)$ a germ of a (linear) free divisor such that $\Der_{x_0}(-\log D)$ is generated by $\delta_1, \dots, \delta_n$. Let $\delta'_i=\sum_{j=1}^n a_i^j(x,t)\partial / \partial x_j$, $i=1, \dots, n$, be a system of holomorphic vector fields at $(x_0, s)\in (\C^n \times S,(0,s))$ such that
\begin{enumerate}
	\item\label{1r} $\delta'_i|_{\C^n,x_0}=\delta_i$ for all $i=1, \dots, n$;
	\item\label{2r} $[\delta'_i,\delta'_j] \in \sum_{k=1}^n \mathcal{O}_{\C^n \times S, (x_0, s)} \delta'_k$ for $i,j=1, \dots, n$;
	\item\label{3r} $\det(a_i^j)=h$ defines a reduced hypersurface $X$.
\end{enumerate}
Then for $X=\{ h(x,t)=0\}$, $\delta'_1, \dots, \delta'_n$ belongs to $\Der_{(x_0, s)}(-\log X/S)$, $\{ \delta'_1, \dots, \delta'_n\}$ is a free basis of $\Der_{(x_0, s)}(-\log X/S)$ and $X$ is a (linearly) admissible deformation of $(D,0)$ over $(S,s)$.
\end{Lemma}
\begin{proof} First of all we need to show that each $\delta_k'\in\Der_{(x_0, s)}(-\log X/S)$. We have the following equalities
$$\delta_k'(h)=\delta_k'(\det[\delta_1',\dots,\delta_n'])=\sum_{j=1}^n\det[\delta_1'\dots,\delta_{j-1}',\delta_k'(\delta_j'),\delta_{j+1}',\dots,\delta_n']=$$
$$=\sum_{j=1}^n\det[\delta_1'\dots,\delta_{j-1}',[\delta_k',\delta_j']+\delta_j'(\delta_k'),\delta_{j+1}',\dots,\delta_n']= $$
$$=\sum_{j=1}^n\det[\delta_1'\dots,\delta_{j-1}',[\delta_k',\delta_j'],\delta_{j+1}',\dots,\delta_n']+\sum_{j=1}^n\det[\delta_1'\dots,\delta_{j-1}',\delta_j'(\delta_k'),\delta_{j+1}',\dots,\delta_n'].$$
By \ref{2r}, $\det[\delta_1'\dots,\delta_{j-1}',[\delta_k',\delta_j'],\delta_{j+1}',\dots,\delta_n']\in(h)\mathcal{O}_{\C^n\times S,(x_0,s)}$ for all $j=1,\dots,n$, and so the first part of the last equality is in $(h)\mathcal{O}_{\C^n\times S,(x_0,s)}$. Furthermore, if we consider the matrices $A=[\delta_1',\dots,\delta_n']$ and $B=(\partial a_k^i/\partial x_j)_{i,j=1,\dots,n}$, we can apply Lemma \ref{molteachcolumntrace} and obtain
$$\sum_{j=1}^n\det[\delta_1'\dots,\delta_{j-1}',\delta_j'(\delta_k'),\delta_{j+1}',\dots,\delta_n']=\sum_{i=1}^n\frac{\partial a_k^i}{\partial x_i}h\in(h)\mathcal{O}_{\C^n\times S,(x_0,s)}.$$
This shows that $\delta_k'(h)\in(h)\mathcal{O}_{\C^n\times S,(x_0,s)}$ and so $\delta_k'\in\Der_{(x_0, s)}(-\log X/S)$, for all $k=1,\dots,n$.

Notice now that by \ref{1r} and \ref{3r}, $h|_{\C^n,x_0}=f$. Moreover, by \ref{1r}
$$\Der_{x_0}(-\log D)\subset \Der_{(x_0,s)}(-\log X/S)/\mathfrak{m}_{S,s}\Der_{(x_0,s)}(-\log X/S).$$ Consider $\sigma\in\Der_{(x_0, s)}(-\log X/S)$ such that $\sigma|_{\C^n,x_0}\notin\Der_{x_0}(-\log D)$. But $\sigma(h)=\alpha h$ for some $\alpha\in\mathcal{O}_{\C^n\times S,(x_0, s)}$. Hence $(\sigma(h))|_{\C^n,x_0}=\sigma|_{\C^n,x_0}(f)=\alpha|_{\C^n,x_0}f$ and so $\sigma|_{\C^n,x_0}\in\Der_{x_0}(-\log D)$, but this is a contradiction. Hence $\Der_{x_0}(-\log D)=\Der_{(x_0,s)}(-\log X/S)/\mathfrak{m}_{S,s}\Der_{(x_0,s)}(-\log X/S)$ and so $X$ is a (linearly) admissible deformation of $(D,0)$ over $(S,s)$. 

Consider $\sigma \in \Der_{(x_0, s)}(-\log X/S)$. Then we want to prove that $\sigma \in \sum_{i=1}^n \mathcal{O}_{\C^n \times S,(x_0, s)} \delta'_i$. By Cramer's rule, $h \partial/\partial x_j \in  \sum_{i=1}^n \mathcal{O}_{\C^n \times S, (x_0, s)} \delta'_i$ for all $j=1, \dots, n$, hence we can consider $h\sigma =  \sum_{i=1}^n f_i \delta'_i$, for some $f_i \in \mathcal{O}_{\C^n \times S, (x_0, s)}$. By Lemma \ref{detrelvect}, we have that $\det[\delta'_1, \dots, \delta'_{i-1}, \sigma, \delta'_{i+1}, \dots, \delta'_n] \in (h)\mathcal{O}_{\C^n \times S, (x_0, s)}$. Thus $$h\det[\delta'_1, \dots, \delta'_{i-1}, \sigma, \delta'_{i+1}, \dots, \delta'_n]$$ $$=\det[\delta'_1, \dots, \delta'_{i-1}, h\sigma, \delta'_{i+1}, \dots, \delta'_n]$$ $$= \det[\delta'_1, \dots, \delta'_{i-1}, f_i\delta'_i, \delta'_{i+1}, \dots, \delta'_n]$$ $$= f_i\det[\delta'_1, \dots, \delta'_n] = f_i h \in (h^2)\mathcal{O}_{\C^n \times S, (x_0, s)}.$$ Thus $f_i \in (h)\mathcal{O}_{\C^n \times S, (x_0, s)}$ for all $i$. This show that $\sigma =\sum_{i=1}^n (f_i/h) \delta'_i \in  \sum_{i=1}^n \mathcal{O}_{\C^n \times S, (x_0, s)} \delta'_i$. \end{proof} 

Notice that if we consider $S$ to be a reduced point, then the previous Lemma is the same statement of Lemma \ref{saitolemma}.

We can now state and prove the main result of the section:

 \begin{Theorem}\label{infdefref} Let $(D,0)=(V(f),0) \subset (\C^n,0)$ be a germ of a free divisor and $\delta_1, \dots , \delta_n$ a set of generators for $\Der(-\log D)$. Any element of $\mathcal{FT}^1(D)$ can be represented by $n$ classes %To find an infinitesimal admissible deformation of $(D,0)$ it is enough to find $n$ vector fields 
 $\tilde{\delta}_1, \dots , \tilde{\delta}_n \in \Der_{\C^n}/\Der(-\log D)$ such that the $\mathcal{O}_{\C^n \times T_\epsilon,0}$-module generated by $\delta'_1=\delta_1+\epsilon \cdot \tilde{\delta}_1, \dots , \delta'_n=\delta_n +\epsilon \cdot\tilde{ \delta}_n$ is closed under Lie brackets. If the deformation is linearly admissible, then the coefficients of all $\tilde{\delta}_i$, in any representation of an element of $\mathcal{FT}^1(D)$, must be linear functions too.
 \end{Theorem}
 \begin{proof} %By Lemma \ref{relatversaitlem}, if we have $n$ vector fields $\tilde{\delta}_1, \dots , \tilde{\delta}_n \in \Der_{\C^n}$ such that the $\mathcal{O}_{\C^n \times T_\epsilon}$-module generated by $\delta_1+\epsilon \cdot \tilde{\delta}_1, \dots , \delta_n +\epsilon \cdot \tilde{\delta}_n$ is closed under Lie brackets, we have an admissible deformation of $D$.
 
Let $(X,x)\subset(\C^n\times T_\epsilon,0)$ be an infinitesimal (linearly) admissible deformation of $(D,0)$. By Remark \ref{forminfdef}, it is of the form $(X,0)=(V(f+ \epsilon \cdot f'),0) \subset (\C^n \times T_\epsilon,0)$. By Proposition \ref{generalizsaitocrit1}, the fact that $(X,0)$ is the total space of an infinitesimal (linearly) admissible deformation of $(D,0)$ implies that there exists an $n \times n$ matrix $A(\epsilon)$ with coefficients in $\C[x_1, \dots , x_n, \epsilon]/(\epsilon^2)$ such that $\det A(\epsilon)=(f+ \epsilon \cdot f')$. But $\epsilon^2=0$ implies that we can write $A(\epsilon)= B+\epsilon \cdot C$, where $B$ and $C$ are  $n \times n$ matrices with coefficients in  $\C[x_1, \dots , x_n]$. Hence $f=\det A(0)= \det B$ and so $B$ is a Saito matrix for $(D,0)$. We can then take $\delta_i$ as the columns of $B$ and $\tilde{\delta}_i$ as the columns of $C$ and this proves that the Lie algebra $\Der(-\log X/T_\epsilon)$ is generated by $\delta_1+\epsilon \cdot \tilde{\delta}_1, \dots , \delta_n +\epsilon \cdot \tilde{\delta}_n$ as required. Because $\Der(-\log X/T_\epsilon)$ is a Lie algebra, then $[\delta_i',\delta_j']\in\Der(-\log X/T_\epsilon)$ for all $i,j=1,\dots,n$, but then $[\delta'_i,\delta'_j] \in \sum_{k=1}^n \mathcal{O}_{\C^n \times S, (x_0, s)} \delta'_k$ for $i,j=1, \dots, n$.

We now consider the classes of $\tilde{\delta}_1, \dots,\tilde{\delta}_n$ modulo $\Der(-\log D)$, because if $\tilde{\delta}_1, \dots,\tilde{\delta}_n\in\Der(-\log D)$, then $f'\in(f)\mathcal{O}_{\C^n,0}$ and hence, by \cite{intdef}, Chapter II, 1.4, the deformation is trivial.

On the other hand, let $\tilde{\delta}_1, \dots , \tilde{\delta}_n \in \Der_{\C^n}/\Der(-\log D)$ be $n$ classes of vector fields such that  the $\mathcal{O}_{\C^n \times T_\epsilon,0}$-module generated by $\delta_1+\epsilon \cdot \tilde{\delta}_1, \dots , \delta_n +\epsilon \cdot \tilde{\delta}_n$ is closed under Lie brackets. The determinant of the matrix of coefficients $[\delta_1+\epsilon \cdot \tilde{\delta}_1, \dots , \delta_n +\epsilon \cdot \tilde{\delta}_n]$ is equal to $f+ \epsilon \cdot f'$ and so by Lemma \ref{relatversaitlem} it is enough to show that this determinant is reduced. First, noticed that for $\epsilon=0$ the determinant is equal to $f$ and hence is reduced. Now, reducedness is an open property and so the result holds.

%Moreover, $f+ \epsilon \cdot f'$ is reduced if and only if $(f+ \epsilon \cdot f')=\sqrt{(f+ \epsilon \cdot f')}$ if and only if $\sqrt{(f+ \epsilon \cdot f')}/(f+ \epsilon \cdot f')=0$ but this is a coherent sheaf and so its support is closed and hence for small $\epsilon$ we have that $f+ \epsilon \cdot f'$ is reduced.\\
The last part of the statement is trivial. \end{proof}
%More generally, we have the following:
%\begin{Proposition}\label{deformref} Let $(D,0)\subset (\C^n,0)$ be a germ of a free divisor, let $\delta_1, \dots , \delta_n$ be a set of generators for $\Der(-\log D)$ and, let $(S,s)$ be a $r$ dimensional complex space germ. To find a (linearly) admissible deformation of $(D,0)$ over $(S,s)$ it is enough to find a family of classes of vector fields $\tilde{\delta}_{1,i}, \dots , \tilde{\delta}_{n,i} \in \Der_{\C^n}/\Der(-\log D)$, $i \in \mathbb{N}^r$ and $i \ne (0,\dots,0)$, such that the $\mathcal{O}_{\C^n \times S,(0,s)}$-module generated by: $$\bigl(\delta_1+\sum_{\substack{i \in \mathbb{N}^r \\ i \ne (0,\dots,0)}} t^i \cdot \tilde{\delta}_{1,i}\bigr), \dots , \bigl(\delta_n +\sum_{\substack{i \in \mathbb{N}^r \\ i \ne (0,\dots,0)}}t^i \cdot \tilde{\delta}_{n,i}\bigr)$$ is closed under Lie brackets. Moreover, if  $(D,0)$ is linear, then the coefficients of all $\tilde{\delta}_{k,i}$ need to be linear functions too.\end{Proposition}
%\begin{proof} Similar to the proof of Theorem \ref{infdefref}.\end{proof}

\subsection{The complexes $\mathcal{C}^\bullet$ and $\mathcal{C}_0^\bullet$}
%\begin{Remark}  Let $D=V(f)\subset  \mathbb{C}^{n}$ be a (linear) free divisor, then there are well-defined operations:
%$$\Der(-\log D) \times \Der_{\C^n} \to \Der_{\C^n} \text{ defined  by } (g,h) \mapsto [g,h]$$ and
%$$\Der(-\log D) \times \Der(-\log D) \to \Der(-\log D) \text{ defined  by } (g,h) \mapsto [g,h].$$
%\end{Remark}
We recall here the notion of the complex of Lie algebroid cohomology in the case of $\Der(-\log D)$, see  \cite{MR0154906} for the general theory.

\begin{Definition} Let $\mathcal{C}^\bullet$ be the complex with modules
$$\mathcal{C}^p:=\mathcal{H}om_{\mathcal{O}_{\C^n}}(\bigwedge^p\Der(-\log D), \Der_{\C^n}/\Der(-\log D))$$
and differentials
$$(d^p(\psi))(\delta_1 \wedge \dots \wedge \delta_{p+1}):=\sum_{i=1}^{p+1}(-1)^i[\delta_i,\psi(\delta_1 \wedge \dots \wedge \widehat{\delta_i} \wedge \dots \wedge \delta_{p+1})]  +$$ $$+\\ \sum_{1 \le i<j \le p+1}(-1)^{i+j-1}\psi([\delta_i,\delta_j]\wedge \delta_1 \wedge \dots \wedge \widehat{\delta_i}\wedge \dots \wedge \widehat{\delta_j}\wedge \dots \wedge \delta_{p+1}).$$
\end{Definition}
It is a straightforward computation to check that $d^{p+1}\circ d^p =0$, so $\mathcal{C}^\bullet$ is a complex.

\begin{Remark} Notice that $$\mathcal{C}^0 = \Der_{\C^n}/\Der(-\log D)$$ and the map $d^0$ is defined by $$d^0 \colon \mathcal{C}^0\to \mathcal{H}om_{\mathcal{O}_{\C^n}}(\Der(-\log D), \Der_{\C^n}/\Der(-\log D)) $$ $$\sigma \mapsto (\delta \mapsto [\delta,\sigma]).$$
\end{Remark}

We recall now the definition of the complex of Lie algebra cohomology from \cite{hochserre}.
\begin{Definition} Let $\mathcal{C}^\bullet_0$ be the complex defined by $$\mathcal{C}^p_0:=Hom_{\C}(\bigwedge^p\Der(-\log D)_0, (\Der_{\C^n}/\Der(-\log D))_0)$$ and the differentials $$(d^p_0(\psi))(\delta_1 \wedge \dots \wedge \delta_{p+1}):=\sum_{i=1}^{p+1}(-1)^i[\delta_i,\psi(\delta_1 \wedge \dots \wedge \hat{\delta_i} \wedge \dots \wedge \delta_{p+1})]  +$$ $$+\\ \sum_{1 \le i<j \le p+1}(-1)^{i+j-1}\psi([\delta_i,\delta_j]\wedge \delta_1 \wedge \dots \wedge \hat{\delta_i}\wedge \dots \wedge \hat{\delta_j}\wedge \dots \wedge \delta_{p+1})$$
where $(\Der_{\C^n}/\Der(-\log D))_0$ is the weight zero part of $\Der_{\C^n}/\Der(-\log D)$.
\end{Definition}
$(\mathcal{C}^\bullet_0, d^\bullet_0)$ is a well defined complex because it has the same differentials as the complex $(\mathcal{C}^\bullet, d^\bullet)$ and because $\Der(-\log D)_0$ is a Lie subalgebra of $\Der(-\log D)$.

\subsection{Infinitesimal admissible deformations}

%\begin{Definition} Let $(D,0)=(V(f),0)\subset  (\C^n,0)$ be a germ of a free divisor, then the space of isomorphism classes of infinitesimal admissible deformations of $(D,0)$ modulo the trivial ones will be denoted by $\mathcal{FT}^1(D).$
%\end{Definition}
\begin{Theorem}\label{FT1} Let $(D,0)\subset  (\C^n,0)$ be a germ of a free divisor. Then the germ at the origin of the first cohomology sheaf of the complex $\mathcal{C}^\bullet$ is isomorphic to $\mathcal{FT}^1(D)$, i.e. $ \mathcal{H}^1(\mathcal{C}^\bullet)_0\cong \mathcal{FT}^1(D).$
\end{Theorem}
\begin{proof} To prove that we can identify $\mathcal{H}^1(\mathcal{C}^\bullet)_0$ with $\mathcal{FT}^1(D)$, two things have to be checked: we must first identify the elements of $\ker(d^1 \colon \mathcal{C}^1 \to \mathcal{C}^2)$ with admissible deformations of $(D,0)$. Then, we have to show that the image of $d^0 \colon \mathcal{C}^0 \to \mathcal{C}^1$ is the collection of trivial admissible deformations of $(D,0)$.

By Proposition \ref{infdefref}, we are looking for $n$ classes of vector fields $\tilde{\delta}_1, \dots , \tilde{\delta}_n \in \Der_{\C^n}/\Der(-\log D)$ such that the $\mathcal{O}_{\C^n\times T_\epsilon,0}$-module generated by the elements $\delta_1+\epsilon \cdot \tilde{\delta}_1, \dots , \delta_n +\epsilon \cdot \tilde{\delta}_n$ is closed under Lie brackets.

Take an element $\psi \in \ker(d^1)$, which means that $$\psi([\delta,\nu])-[\delta,\psi(\nu)]+[\nu,\psi(\delta)]=0 \text{ in } \Der_{\C^n}/\Der(-\log D)$$ for all $\delta,\nu \in \Der(-\log D)$. Then $\psi$ corresponds to the admissible deformation given by the $\mathcal{O}_{\C^n\times T_\epsilon,0}$-module $\mathcal{L}$ generated by $$\delta_1+\epsilon \cdot \psi(\delta_1), \dots , \delta_n +\epsilon \cdot \psi(\delta_n) .$$  By $\C$-linearity of the Lie brackets, $\mathcal{L}$ is closed under Lie brackets if and only if for any two elements $\delta+\epsilon \cdot \psi(\delta), \nu+\epsilon \cdot \psi(\nu) \in \mathcal{L}$ we have $[\delta+\epsilon \cdot \psi(\delta), \nu+\epsilon \cdot \psi(\nu)] \in \mathcal{L}$, which is equivalent to $$F:=[\delta,\nu]+\epsilon \cdot([\delta,\psi(\nu)]-[\nu,\psi(\delta)])\in \mathcal{L}.$$ Consider $G:=[\delta,\nu]+\epsilon \cdot \psi([\delta,\nu])$ which is an element of $\mathcal{L}$, so the condition $F \in \mathcal{L}$ is equivalent to $G-F \in \mathcal{L}$, that is $$ \psi([\delta,\nu])-[\delta,\psi(\nu)]+[\nu,\psi(\delta)] \in \Der(-\log D).$$ This means exactly that $\psi \in \ker(d^1)$. 

Let us consider now an infinitesimal admissible deformation $(X,0)=(V(f+ \epsilon \cdot f'),0)$. Then by the previous part of the proof, $\Der(-\log X/T_{\epsilon}) = \langle \delta_1+\epsilon \cdot \psi(\delta_1), \dots , \delta_n +\epsilon \cdot \psi(\delta_n) \rangle$ for some $\psi \in \ker(d^1)$. By Remark \ref{trivial}, $f+ \epsilon \cdot f'$ is trivial if and only if $(\varphi ^\ast f)=(f(x+\epsilon \cdot \sigma(x))) = (f+ \epsilon \cdot f')$, for some $\varphi\in\Aut(\C^n\times T_\epsilon)$. In this situation, the module of vector fields generated by $\varphi^\ast(\Der(-\log D))$ is equal to $\Der(-\log X/T_\epsilon)$, i.e. 
$$\langle D_{\varphi^{-1}(x)}\varphi(\delta_1(\varphi^{-1}(x))), \dots, D_{\varphi^{-1}(x)}\varphi(\delta_n(\varphi^{-1}(x))) \rangle  = \langle \delta_1+\epsilon \cdot \psi(\delta_1), \dots , \delta_n +\epsilon \cdot \psi(\delta_n) \rangle,$$
where for $h\colon X\to Y$, then $D_xh\colon T_xX\to T_{h(x)}Y$ is the tangent map. 
Because we can consider each vector field on $\C^n$ also as a map from $\C^n$ into itself, we have the following equalities 
$$D_{\varphi^{-1}(x)}\varphi(\delta_i(\varphi^{-1}(x))) = D_{x-\epsilon \cdot \sigma(x)}\varphi(\delta_i(x-\epsilon \cdot \sigma(x)))=$$ 
$$=\delta_i(x-\epsilon \cdot \sigma(x))+\epsilon \cdot D_{x-\epsilon \cdot \sigma(x)}\sigma(\delta_i(x-\epsilon \cdot \sigma(x)))=$$ 
$$=\delta_i(x)-\epsilon \cdot (D_x\delta_i(\sigma(x))-D_{x-\epsilon \cdot \sigma(x)}\sigma(\delta_i(x)))=$$ 
$$=\delta_i(x)+\epsilon \cdot (D_x\sigma(\delta_i(x))-D_x\delta_i(\sigma(x)))= \delta_i(x)+\epsilon \cdot [\sigma,\delta_i](x)$$ and that tells us that $\psi(\delta_i)=[\sigma,\delta_i]$, i.e. $\psi \in \image(d^0)$. \end{proof}
\begin{Lemma}\label{selfnorm} Let $D\subset  \C^n$ be a free divisor. Then $\Der(-\log D)$ is a self-normalising Lie subalgebra of $\Der_{\C^n}$. That is, if we consider $\chi \in \Der_{\C^n}$ such that $[\chi,\delta] \in \Der(-\log D)$ for all $\delta \in \Der(-\log D)$, then $\chi \in \Der(-\log D)$.
\end{Lemma}
\begin{proof} By the definition of $\Der(-\log D)$, it is enough to show that if we consider $p \in D$ a smooth point, then $\chi(p) \in T_pD$. Without loss of generality, we can suppose that at $p$ the divisor $D$ is defined by the equation $x_1=0$, that its Saito matrix is \begin{equation*} [\delta_1, \cdots, \delta_n]=
\begin{bmatrix}
		x_1&0&0&\cdots&0\\
		0&1&0&\cdots&0\\
		0&0&1&\cdots&0\\
		\vdots&\vdots&\vdots& &\vdots\\
		0&0&0&\cdots&1\\
\end{bmatrix}
\end{equation*}
and that $\chi(p)=\sum_{i=1}^n a_i\partial/\partial x_i$ with $a_i \in \mathcal{O}_{\C^n,p}$. In this way, we have reduced the problem to proving that $a_1 \in (x_1)\mathcal{O}_{\C^n,p}$.

By hypothesis, $[\chi,\delta] \in \Der_p(-\log D)$ for all $\delta \in \Der_p(-\log D)$, in particular $[\chi,\delta_1]=a_1\partial/\partial x_1 - \sum_{i=1}^n x_1\partial a_i/\partial x_1\partial/\partial x_i=(a_1-x_1\partial a_1/\partial x_1)\partial/\partial x_1- \sum_{i=2}^n x_1\partial a_i/\partial x_1\partial/\partial x_i \in \Der_p(-\log D)$. Hence, $(a_1-x_1\partial a_1/\partial x_1) \in (x_1)\mathcal{O}_{\C^n,p}$ and so $a_1 \in (x_1)\mathcal{O}_{\C^n,p}$ as required.\end{proof}
In a similar way we can prove the following:
\begin{Lemma}\label{selfnormzero}  Let $D\subset  \C^n$ be a linear free divisor. Then $\Der(-\log D)_0$ is a self-normalising Lie subalgebra of $(\Der_{\C^n})_0$.
\end{Lemma}
\begin{Proposition}\label{zerocohomol} $\mathcal{H}^0(\mathcal{C}^\bullet)=0$.
\end{Proposition}
\begin{proof} Consider $\sigma \in \mathcal{H}^0(\mathcal{C}^\bullet)=\ker(d^0)$. Hence, $[-,\sigma]$ is the zero map, i.e. for all $\delta \in \Der(-\log D)$ we have that $[\delta,\sigma]\in \Der(-\log D)$. Then by Lemma \ref{selfnorm}, $\sigma \in \Der(-\log D)$.\end{proof}
\begin{Proposition} Let $(D,0)\subset(\C^n,0)$ be a germ of a smooth divisor. Then $\mathcal{FT}^1(D)=0$. 
\end{Proposition}
\begin{proof} We can suppose $f=x_1$ and we can take as Saito matrix the matrix
 \begin{equation*} S=[\delta_1,\dots,\delta_n]=
\begin{bmatrix}
		x_1&0&0&\cdots&0\\
		0&1&0&\cdots&0\\
		0&0&1&\cdots&0\\
		\vdots&\vdots&\vdots& &\vdots\\
		0&0&0&\cdots&1\\
\end{bmatrix}.
\end{equation*}
Moreover, we can represent an element of $\mathcal{C}^1$ as the column of the $n\times n$ matrix $S+\epsilon \cdot T$, where $T$ is the matrix
 \begin{equation*} T=[\tilde{\delta}_1,\dots,\tilde{\delta}_n]=
\begin{bmatrix}
		g_1&g_2&\cdots&g_n\\
		0&0&\cdots&0\\
		\vdots&\vdots& &\vdots\\
		0&0&\cdots&0\\
\end{bmatrix} 
\end{equation*}
and $g_i=g_i(x_2,\dots,x_n)\in \mathcal{O}_{\C^n,0}$. 

Because $[\delta_i,\delta_j]=0$ for every $i,j=1,\dots,n$, then the element $S+\epsilon \cdot T$ is in the kernel of $d^1$ if and only if $g_i=-\partial g_1/\partial x_i$ for all $i=2,\dots,n$. To show that this element is zero in cohomology, it is enough to find $\sigma\in \mathcal{C}^0= \Der_{\C^n}/\Der(-\log D)$ such that $[\sigma,\delta_i]=\tilde{\delta}_i$ for all $i=1,\dots,n$, i.e. $S+\epsilon \cdot T$ is in the image of $d^0$. Consider $\sigma=g_1\partial/\partial x_1$, then it is the element we are looking for.
\end{proof}
\begin{Proposition} Let $(D,0)\subset (\C^n,0)$ be the germ of the normal crossing divisor. Then $\mathcal{FT}^1(D)=0$. 
\end{Proposition}
\begin{proof} Let $f=x_1\cdots x_n$ be a defining equation for $D$.
We can take as Saito matrix
 \begin{equation*} S=[\delta_1,\dots,\delta_n]=
\begin{bmatrix}
		x_1&0&\cdots&0\\
		0&x_2&\cdots&0\\
		\vdots&\vdots& &\vdots\\
		0&0&\cdots&x_n\\
\end{bmatrix}.
\end{equation*}
Moreover, we can represent an element of $\mathcal{C}^1$ as columns of the $n\times n$ matrix $S+\epsilon \cdot T$, where $T$ is the matrix
 \begin{equation*} T=[\tilde{\delta}_1,\dots,\tilde{\delta}_n]=
\begin{bmatrix}
		g_{1,1}&g_{1,2}&\cdots&g_{1,n}\\
		g_{2,1}&g_{2,2}&\cdots&g_{2,n}\\
		\vdots&\vdots& &\vdots\\
		g_{n,1}&g_{n,2}&\cdots&g_{n,n}\\
\end{bmatrix} 
\end{equation*}
and $g_{i,j}=g_{i,j}(x_1,\dots,\hat{x}_i,\dots,x_n)\in \mathcal{O}_{\C^n}$. 

Because $[\delta_i,\delta_j]=0$ for every $i,j=1,\dots,n$, then the element represented by $S+\epsilon \cdot T$ is in the kernel of $d^1$ if and only if $A_{i,j}=-[\delta_i,\tilde{\delta}_j]+[\delta_j,\tilde{\delta}_i]\in\Der(-\log D)$ for all $i,j=1,\dots,n$. Let us suppose that $i<j$, then
\begin{equation*} A_{i,j}=
\begin{bmatrix}
	-x_i\partial g_{1,j}/\partial x_i\\
	\vdots \\
	g_{i,j}\\
	\vdots\\
	-x_i\partial g_{j,j}/\partial x_i\\
	\vdots\\
	-x_i\partial g_{n,j}/\partial x_i\\
\end{bmatrix} +
\begin{bmatrix}
	x_j\partial g_{1,i}/\partial x_j\\
	\vdots \\
	x_j\partial g_{i,i}/\partial x_j\\
	\vdots\\
	-g_{j,i}\\
	\vdots\\
	x_j\partial g_{n,i}/\partial x_j\\
\end{bmatrix} 
\end{equation*}
Now, $A_{i,j}\in \Der(-\log D)$ for all $i,j=1,\dots,n$ if and only if $A_{i,j}=0$ if and only if
 \begin{equation*} T=
\begin{bmatrix}
		g_{1,1}&-x_2\partial g_{1,1}/\partial x_2&\cdots&-x_n \partial g_{1,1}/\partial x_n\\
		-x_1\partial g_{2,2}/\partial x_1&g_{2,2}&\cdots&-x_n \partial g_{2,2}/\partial x_n\\
		\vdots&\vdots& &\vdots\\
		-x_1\partial g_{n,n}/\partial x_1&-x_2 \partial g_{n,n}/\partial x_2&\cdots&g_{n,n}\\
\end{bmatrix}. 
\end{equation*}
 To show that this element is zero in cohomology, it is enough to find $\sigma\in \mathcal{C}^0= \Der_{\C^n}/\Der(-\log D)$ such that $[\sigma,\delta_i]=\tilde{\delta}_i$ for all $i=1,\dots,n$, i.e. $S+\epsilon \cdot T$ is in the image of $d^0$. Consider
 \begin{equation*} \sigma=
\begin{bmatrix}
	g_{1,1}\\
	\vdots \\
	g_{n,n}\\
\end{bmatrix}
\end{equation*}
then it is the element we are looking for.
\end{proof}
\begin{Remark} There exist free divisors such that $\mathcal{FT}^1(D)\ne0$. 
\end{Remark}
\begin{proof} Consider $f=xy(x-y)(x+y)\in \C[x,y]$ and the germ of a free divisor $(D,0)=(V(f),0)\subset (\C^2,0)$ with Saito matrix \begin{equation*} A=
\begin{bmatrix}
		x&0\\
		y&x^2y-y^3\\
\end{bmatrix}.
\end{equation*}
To find an infinitesimal admissible deformation for $(D,0)$ we have to find a non zero element  $\alpha \in \mathcal{H}^1(\mathcal{C}^\bullet)_0=\mathcal{FT}^1(D).$ Let $\alpha$ be defined by the columns of the following matrix \begin{equation*} B=
\begin{bmatrix}
		0&0\\
		0&xy^2-y^3\\
\end{bmatrix}.
\end{equation*}
this is an element of $\mathcal{H}^1(\mathcal{C}^\bullet)_0$ that describes the infinitesimal admissible deformation $X=V(xy(x-y)(x+(1+\epsilon)y))=V(f+\epsilon(x^2y^2-xy^3)) \subset \C^2\times T_\epsilon$. This infinitesimal admissible deformation is non-trivial because it is a non-trivial deformation of $f$ as a germ of function because $x^2y^2-xy^3$ is not in the Jacobian ideal of $f$, see \cite{intdef}, Chapter II, 1.4. \end{proof}

\subsection{Infinitesimal linearly admissible deformations}
%\begin{Definition}Let $(D,0)=(V(f),0)\subset  (\C^n,0)$ be a germ of a linear free divisor, then the space of isomorphism classes of infinitesimal linearly admissible deformations of $(D,0)$ modulo the trivial ones of will be denoted by $\mathcal{LFT}^1(D).$
%\end{Definition}
\begin{Theorem} Let $(D,0)\subset  (\C^n,0)$ be a germ of a linear free divisor. Then the germ at the origin of the first cohomology sheaf of the complex $\mathcal{C}^\bullet_0$ is isomorphic to $\mathcal{LFT}^1(D)$, i.e. $ H^1(\mathcal{C}^\bullet_0)_0\cong \mathcal{LFT}^1(D).$
\end{Theorem}
\begin{proof} This is a consequence of Theorem \ref{FT1} and the second part of Theorem \ref{infdefref}. \end{proof}
\begin{Corollary} Let $(D,0)\subset (\C^n,0)$ be a germ of a linear free divisor. Then the functor $\mathbf{LFD}_D$ satisfies Schlessinger condition (H3) from \cite{schless}.
\end{Corollary}
\begin{proof} This is a consequence of the previous Theorem and of the fact that the cohomology of a finite dimensional Lie algebra is finite dimensional.
\end{proof}
\begin{Corollary} Let $(D,0)\subset (\C^n,0)$ be a germ of a linear free divisor. Then $\mathbf{LFD}_D$ has a hull.
\end{Corollary}
\begin{proof} This is a consequence of Theorem 2.11 from \cite{schless}, Theorem \ref{isdeffunct} and the previous Corollary. 
\end{proof}
\begin{Proposition} $H^0(\mathcal{C}^\bullet_0)=0$.
\end{Proposition}
\begin{proof} Like the proof of Proposition \ref{zerocohomol} but using Lemma \ref{selfnormzero}.\end{proof}
\begin{Definition} Let $M$ be a vector space and let $\mathfrak{g}$ be a Lie algebra. A \emph{representation of $\mathfrak{g}$ in $M$} is a homomorphism $\varrho$ of $\mathfrak{g}$ in $\mathfrak{gl}(M)$.
\end{Definition}

In what follows, we will refer both to the homomorphism $\varrho$ and to the vector space $M$ as representations of $\mathfrak{g}$.

\begin{Remark} $\mathcal{LFT}^1(D)$ is the first Lie algebra cohomolgy of $\Der(-\log D)_0$ with coefficients in the non-trivial representation $(\Der_{\C^n}/\Der(-\log D))_0$. 
\end{Remark}

We collect now some results from \cite{dix}, \cite{hochserre} and \cite{wei} about Lie algebras and Lie algebra cohomology, that will allow us to compute $\mathcal{LFT}^1(D)$ more easily in the case of germs of reductive linear free divisors.
\begin{Proposition}\label{reducrepr}\emph{(\cite{dix}, Corollary 1.6.4)} Let $\mathfrak{g}$ be a reductive Lie algebra and let $\varrho$ be a finite dimensional representation of $\mathfrak{g}$. Then the following condition are equivalent
\begin{enumerate}
	\item $\varrho$ is semisimple;
	\item for all $a$ in the centre of $\mathfrak{g}$, $\varrho(a)$ is semisimple.
\end{enumerate}
\end{Proposition}
We will use the following celebrated theorem of Hochschild and Serre
\begin{Theorem}\label{isocoho} \emph{(\cite{hochserre}, Theorem 10)} Let $\mathfrak{g}$ be a reductive Lie algebra of finite dimension over $\C$. Let $M$ be a finite dimensional semisimple representation of $\mathfrak{g}$ such that $M^{\mathfrak{g}}=(0)$, where $M^{\mathfrak{g}}$ is the submodule of $M$ on which $\mathfrak{g}$ acts trivially. Then $H^n(\mathfrak{g},M)=0$ for all $n\ge 0$.
\end{Theorem}
%\begin{Corollary}\label{coreduct} Let $\mathfrak{g}$ be a finite dimensional complex reductive Lie algebra and $M$ be a semisimple representation of $\mathfrak{g}$. Then $H^n(\mathfrak{g},M)=H^n(\mathfrak{g},M^{\mathfrak{g}})$ for all $n\ge 0$.
%\end{Corollary}
%\begin{proof} It is a direct consequence of Theorem \ref{isocoho} and the functoriality of $H^{\bullet}$. \end{proof}
%\begin{Proposition}\label{isomcohomred}\emph{(\cite{wei}, Corollary 7.4.8)} Let $\mathfrak{g}$ be a finite dimensional complex Lie algebra and $M$ be a representation of $\mathfrak{g}$. Suppose that M is trivial, then $H^1(\mathfrak{g},M)=Hom_{\C}(\mathfrak{g}^{ab},M)$, where $\mathfrak{g}^{ab}=\mathfrak{g}/[\mathfrak{g},\mathfrak{g}]= \text{centre of } \mathfrak{g}$.
%\end{Proposition}
%\begin{Proposition} Let $\mathfrak{g}$ be a finite dimensional complex reductive Lie algebra and $M$ be a representation of $\mathfrak{g}$. Suppose that  $M$ is semisimple, then $H^n(\mathfrak{g},M^{\mathfrak{g}})=H^n(\mathfrak{g},\C)\otimes_{\C}M^{\mathfrak{g}}$ for all $n\ge 0$.
%\end{Proposition}
In order to apply the previous theorem, we need the following:
\begin{Lemma}\label{centrelfd} Let $D\subset\C^n$ be a reductive linear free divisor. Then all the elements in the centre of $\Der(-\log D)_0$ are diagonalizable.
\end{Lemma}
\begin{proof} By definition $\mathfrak{g}_D=\{A ~| ~xA^t\partial^t \in \Der(-\log D)_0 \}$ is a reductive Lie algebra and hence by Lemma \ref{reductgroup} and by Lemma 3.6, (2) of \cite{grmondsch},  $G^\circ_D$ is a reductive Lie group. Hence by definition, the centre $Z_{G^\circ_D}$ of $G^\circ_D$ is composed of semisimple transformations. Moreover, the Lie algebra of the identity component of $Z_{G^\circ_D}$ coincides with $Z_{\mathfrak{g}_D}$ the centre of $\mathfrak{g}_D$ and hence it is composed of diagonalizable elements.\end{proof}
\begin{Proposition}\label{reductionlfd} Let $D\subset\C^n$ be a reductive linear free divisor. Then the representation of $\Der(-\log D)_0$ in $(\Der_{\C^n}/\Der(-\log D))_0$ is semisimple.
\end{Proposition}
\begin{proof} This is a consequence of Proposition \ref{reducrepr} and Lemma \ref{centrelfd}.\end{proof}
%\begin{Theorem}\label{reductionlfd} Let $D$ be a reductive linear free divisor, then $\mathcal{LFT}^1(D)= Hom_{\C}(\Der(-\log~D)_0^{ab},(\Der_{\C^n}/\Der(-\log D))_0^{\Der(-\log D)_0})$.
%\end{Theorem}
%\begin{proof} By Proposition \ref{reducrepr} and Lemma \ref{centrelfd} the representation of $\Der(-\log D)_0$ in $(\Der_{\C^n}/\Der(-\log D))_0$ is semisimple. Hence, we can conclude by Corollary \ref{coreduct} and Proposition \ref{isomcohomred}.\end{proof}
\begin{Theorem}\label{zerot1} Let $(D,0)\subset(\C^n,0)$ be a germ of a reductive linear free divisor. Then $\mathcal{LFT}^1(D)=0$.
\end{Theorem}
\begin{proof} By Lemma \ref{selfnormzero}, $(\Der_{\C^n}/\Der(-\log D))_0^{\Der(-\log D)_0}=0$ and hence by Theorem \ref{isocoho}, $\mathcal{LFT}^1(D)=0$.\end{proof}
\begin{Corollary} Let $(D,0)\subset(\C^n,0)$ be a germ of a reductive linear free divisor. Then it is formally rigid.
\end{Corollary}
The statement of Theorem \ref{zerot1} is false if we consider non-reductive germs of linear free divisors. In fact, Brian Pike suggested us the following
\begin{Example}Consider $f=x_5(x_4^4-2x_5x_4^2x_3+x_5^2x_3^2+2x_5^2x_4x_2-2x_5^3x_1)\in\C[x_1,\dots, x_5]$ as a defining equation of the germ of a linear free divisor $(D,0)\subset(\C^5,0)$. Then we can consider the Saito matrix
 \begin{equation*} 
\begin{bmatrix}
		x_4&x_3&x_2&x_1&0\\
		x_5&x_4&0&0&x_2\\
		0&x_5&2x_4&-x_3&2x_3\\
		0&0&x_5&-2x_4&3x_4\\
		0&0&0&-3x_5&4x_5\\
\end{bmatrix}.
\end{equation*}
Consider $\sigma=16x_1\partial/\partial x_1+11x_2\partial/\partial x_2+6x_3\partial/\partial x_3+x_4\partial/\partial x_4-4x_5\partial/\partial x_5$, then $\sigma\in\Ann(D)$ and $\trace(\sigma)=30$, hence, by Lemma \ref{reductrace}, $(D,0)$ is the germ of a non-reductive linear free divisor.

To find an infinitesimal linearly admissible deformation for $(D,0)$ we have to find a non-zero element  $\alpha \in \mathcal{H}^1(\mathcal{C}_0^\bullet)_0=\mathcal{LFT}^1(D).$ Let $\alpha$ be defined by the columns of the following matrix
\begin{equation*} 
\begin{bmatrix}
		0&0&0&0&0\\
		0&0&2x_3&0&0\\
		0&0&-2x_4&0&0\\
		0&0&0&0&0\\
		0&0&0&0&0\\
\end{bmatrix}.
\end{equation*}
This is an element of $\mathcal{H}^1(\mathcal{C}^\bullet_0)_0$ that describes the infinitesimal linearly admissible deformation $X=V(x_5(x_4^4(1+\epsilon)-2x_5x_4^2x_3+x_5^2x_3^2+2x_5^2x_4x_2-2x_5^3x_1))=V(f+\epsilon(x_4^4x_5)) \subset \C^5\times T_\epsilon$. This infinitesimal linearly admissible deformation is non-trivial because it is a non-trivial deformation of $f$ as a germ of function, in fact $x_4^4x_5\notin J(D)$. Moreover, one can check, via a long Macaulay $2$ computation, that $\mathcal{LFT}^1(D)$ is $4$-dimensional and this element is one of its generators. For more details see \cite{Mythesis}, Appendix C.1.
\end{Example}

%\begin{Example}
%\begin{enumerate}
%	\item[i)] Consider the normal crossing divisor of example \ref{ncrdiv}. The centre of $Der(-log~D)_0$ coincide with the all $Der(-log~D)_0$ that is made of diagonal elements, so by Theorem \ref{zerot1}, $\mathcal{LFT}^1(D)=0$.
%	\item[ii)] Consider the divisor $D=V(y^2z^2-4xz^3-4y^3w+18xyzw-27w^2x^2)\subset \C^4$. This is a linear free divisor because we can take the following matrix: \begin{equation*} A=
%\begin{bmatrix}
%		3x&0&y&0\\
%		2y&3x&2z&y\\
%		z&2y&3w&2z\\
%		0&z&0&3w\\
%\end{bmatrix}
%\end{equation*}
%as its Saito matrix. Moreover, we have that $\mathfrak{g}_D=\mathfrak{gl}_2(\C)$ and hence $D$ is a reductive linear free divisor. The centre of $Der(-log~D)_0$ is one dimensional and it is generated by $\sigma=x\partial/\partial x+y\partial/\partial y+z\partial/\partial z+w\partial/\partial w$. In conclusion, also in this case by Theorem \ref{zerot1}, $\mathcal{LFT}^1(D)=0$.
%\end{enumerate}
%\end{Example}

\subsection{The weighted homogeneous case}

\begin{Proposition}\label{t1weighthomogneq} Let $(D,0)\subset(\C^n,0)$ be a germ of a free divisor defined by a weighted homogeneous polynomial of degree $k$. Then an element of $\mathcal{FT}^1(D)$ can be represented by $f'\in \C[x_1,\dots,x_n]_k$, where $\C[x_1,\dots,x_n]_k$ is the space of polynomial of weighted degree $k$. 
\end{Proposition}
\begin{proof} Let $f$ be a defining equation for $(D,0)$. Because $f$ is weighted homogeneous, then there exists $\chi\in\Der(-\log D)$ such that $\chi(f)=f$. 

Consider $(X,x)$ an infinitesimal admissible deformation of $(D,0)$. By Remark \ref{forminfdef}, we can suppose it is defined by the equation $f+\epsilon\cdot f'$, where $f'\in\mathcal{O}_{\C^n,0}$. Suppose that $f'$ is weighted homogeneous of degree $\beta$. Because $(X,x)$ is admissible, it means that $\chi$ lifts and so there exists $\chi'\in\Der_{\C^n}$ such that $(\chi+\epsilon\cdot\chi')(f+\epsilon\cdot f')=(1+\epsilon\cdot\alpha)(f+\epsilon\cdot f')$ and so $\chi'(f)+\chi(f')=\alpha f+f'$, for some $\alpha\in\mathcal{O}_{\C^n,0}$. Because $f'$ is weighted homogeneous of degree $\beta$, then $\chi(f')=\beta f'$. Hence, the previous expression becomes $(\chi'-\alpha)f=(1-\beta)f'$. However, $(\chi'-\alpha)f$ lies in the Tyurina ideal of $f$ which is equal to the Jacobian ideal of $D$ due to the quasi-homogeneity of $f$ and so $(1-\beta)f'$ is in the Jacobian ideal of $D$. If $f'$ is in the Jacobian ideal, then the admissible deformation is trivial, by  \cite{intdef}, Chapter II, 1.4, otherwise $\beta=1$ and so $f'$ is of weighted degree $k$.

If $f'$ is not weighted homogeneous, we can apply the previous argument to each of its weighted homogeneous parts.
\end{proof} 
\begin{Lemma}\label{monomialbasist1} Let $(D,0)\subset(\C^n,0)$ be a germ of a free divisor defined by a weighted homogeneous polynomial. Then a basis of $\mathcal{FT}^1(D)$ can be chosen to be made of monomials.
\end{Lemma}
\begin{proof} This is because we have a good $\C^*$-action.
\end{proof}

\begin{Corollary}\label{t1weighthomogn} Let $(D,0)\subset(\C^n,0)$ be a germ of a free divisor defined by a weighted homogeneous polynomial of degree $k$ with non-zero weights $(a_1,\dots,a_n)$. Then $$\dim_\C\mathcal{FT}^1(D)\le \dim_{\C}\C[x_1,\dots,x_n]_k/J(D)\cap\C[x_1,\dots,x_n]_k,$$ where $J(D)$ is the Jacobian ideal of $D$. 
\end{Corollary}
\begin{proof} It is a consequence of Lemma \ref{monomialbasist1}, Proposition \ref{t1weighthomogneq} and  that $J(D)$ defines only trivial deformations.
\end{proof} 
\begin{Corollary} Let $(D,0)\subset(\C^n,0)$ be a germ of a free divisor defined by a weighted homogeneous polynomial. Then $\mathbf{FD}_D$ has a hull.
\end{Corollary}
\begin{proof} By Corollary \ref{t1weighthomogn}, condition (H3) from \cite{schless} is satisfied. Then the result follows from Theorem \ref{isdeffunct} and Theorem 2.11 from \cite{schless}.
\end{proof}
Because each germ of a linear free divisor $(D,0)\subset(\C^n,0)$ is defined by a homogeneous equation of degree $n$, we have the following:
\begin{Corollary} Let $(D,0)\subset(\C^n,0)$ be a germ of a linear free divisor. Then $\mathbf{FD}_D$ has a hull.
\end{Corollary}
By Corollary \ref{planecurvefree}, every reduced curve is a free divisor. Then:
\begin{Theorem}\label{t1weighthomog2} Let $(D,0)\subset(\C^2,0)$ be a reduced curve germ defined by a weighted homogeneous polynomial of degree $k$. Then $$\mathcal{FT}^1(D)\cong\C[x,y]_k/J(D)\cap\C[x,y]_k.$$ 
\end{Theorem}
\begin{proof} Let $f$ be a defining equation for $(D,0)$. Because $f$ is weighted homogeneous, then there exists $\chi\in\Der_{\C^n}$ such that $\chi(f)=f$. Let $\delta=\partial f/\partial x\partial /\partial y-\partial f/\partial y\partial /\partial x$. Because $D$ has an isolated singularity, then $\delta, \chi$ form a basis of $\Der(-\log D)$.

By Proposition \ref{t1weighthomogneq}, we know that if $(X,x)$ is an infinitesimal admissible deformation of $(D,0)$ defined by $f+\epsilon\cdot f'$, then $f'\in\C[x,y]_k$.

On the other hand, let $f'\in\C[x,y]_k$, then consider $(X,0)$ defined by $f+\epsilon\cdot f'=F$, then it  is an infinitesimal admissible deformation because both $\delta$ and $\chi$  lift. In fact, we can consider $\delta'=\partial F/\partial x\partial /\partial y-\partial F/\partial y\partial /\partial x$ and $\chi$ as elements of $\Der(-\log X/T_\epsilon)$.

We have to go modulo $J(D)\cap\C[x,y]_k$ to avoid trivial admissible deformations.
\end{proof}
\begin{Remark} The previous Theorem is false in higher dimension.
\end{Remark}
\begin{proof} Consider $f=4x^3y^2-16x^4z+27y^4-144xy^2z+128x^2z^2-256z^3\in\C[x,y,z]$. It is weighted homogeneous of degree 12 with weights $(2,3,4)$ and it defines a germ of a free divisor $(D,0)\subset(\C^3,0)$. A Macaulay 2 computation shows that $\dim_\C\C[x,y,z]_{12}/J(D)\cap\C[x,y,z]_{12}=3$ but $\mathcal{FT}^1(D)=0$.
\end{proof}

\begin{Corollary} Let $(D,0)\subset(\C^2,0)$ be a germ of a free divisor defined by a homogeneous polynomial of degree $k$. Then $\dim_\C\mathcal{FT}^1(D)=k-3$ if $k\ge3$, and is zero otherwise.
\end{Corollary}
\begin{proof} If $k=1$, then $J(D)=\C[x,y]$ and if $k=2$, then $J(D)=(x,y)$ and so, by Theorem \ref{t1weighthomog2}, in both cases $\mathcal{FT}^1(D)=0$. 

Let us suppose now that $k\ge3$. We have that $\dim_\C\C[x,y]_k=k+1$ and that $J(D)\cap\C[x,y]_k$ gives us 4 relations: $x\partial f/\partial x, x\partial f/\partial y, y\partial f/\partial x, y\partial f/\partial y$. Because $(D,0)$ is an isolated singularity, then $\partial f/\partial x,\partial f/\partial y$ form a regular sequence and so the Koszul relation generates the relations between the partial derivative of $f$. Because the Koszul relation is of degree $k-1>1$, then $x\partial f/\partial x, x\partial f/\partial y, y\partial f/\partial x, y\partial f/\partial y$ are linearly independent. Hence, $\dim_\C\C[x,y]_k/J(D)\cap\C[x,y]_k=k+1-4=k-3$. We conclude by Theorem \ref{t1weighthomog2}.
\end{proof}
\begin{Example} 
\begin{enumerate}
\item Consider $f=xy(x-y)(x+y)\in \C[x,y]$ and let $(D,0)=(V(f),0)\subset (\C^2,0)$. Then $\mathcal{FT}^1(D)$ is $1$-dimensional and it is generated by $x^2y^2$.
\item Consider $f=x^5+y^4\in \C[x,y]$ and the germ of a free divisor $(D,0)=(V(f),0)\subset (\C^2,0)$. A direct computation shows that $\mathcal{FT}^1(D)=0$ and so it is formally rigid.
\end{enumerate}
\end{Example}
\begin{Remark} Let $(D,0)\subset(\C^n,0)$ be a germ of a free divisor defined by a weighted homogeneous polynomial. Then we can compute the cohomology of $\mathcal{C}^\bullet$ degree by degree, because each module  and map involved is degree preserving.
\end{Remark}
\begin{Theorem} Let $(D,0)\subset(\C^n,0)$ be a germ of a free divisor defined by a weighted homogeneous polynomial. Then $\mathcal{FT}^1(D)\cong (\mathcal{H}^1(\mathcal{C}^\bullet)_0)_0$, where $(\mathcal{H}^1(\mathcal{C}^\bullet)_0)_0$ is the weight zero part of $\mathcal{H}^1(\mathcal{C}^\bullet)_0$.
\end{Theorem}
\begin{proof} Let $f$ be a defining equation for $(D,0)$ weighted homogeneous of degree $k$ and let $(X,x)$ be an infinitesimal admissible deformation of $(D,0)$. By Proposition \ref{t1weighthomogneq}, we can suppose $(X,x)$ has defining equation $f+\epsilon\cdot f'$, with $f'$ weighted homogeneous of degree $k$.

Because $\Der(-\log D)$ is a graded module, we can consider $\delta_1,\dots,\delta_n\in\Der(-\log D)$ a weighted homogeneous basis. By Proposition \ref{generalizsaitocrit1}, $\Der(\log X/T_\epsilon)$ is generated by $\delta_1+\epsilon\cdot\tilde{\delta}_1,\dots,\delta_n+\epsilon\cdot\tilde{\delta}_n$ such that the determinant of their coefficients is $f+\epsilon\cdot f'$. Because $f$ and $f'$ are both weighted homogenous of the same degree, then each $\tilde{\delta}_i$ is weighted homogeneous of the same degree as $\delta_i$, for all $i=1,\dots,n$. 

As seen in the proof of Theorem \ref{FT1}, there exists $\psi\in\mathcal{C}^1$ such that $\psi(\delta_i)=\tilde{\delta}_i$. So by the previous argument $\psi$ is a weight-preserving map and so represents an element of  $(\mathcal{H}^1(\mathcal{C}^\bullet)_0)_0$.
\end{proof}
\begin{Corollary}\label{isot1lfd} Let $(D,0)\subset(\C^n,0)$ be a germ of a linear free divisor. Then $\mathcal{FT}^1(D)\cong\mathcal{LFT}^1(D)$.
\end{Corollary}
\begin{proof} It is clear that $(\mathcal{H}^1(\mathcal{C}^\bullet)_0)_0=\mathcal{H}^1(\mathcal{C}^\bullet_0)_0$.
\end{proof}
\begin{Corollary} Let $(D,0)\subset(\C^n,0)$ be a germ of a reductive linear free divisor. Then it is formally rigid also as free divisor.
\end{Corollary}
\begin{proof} This is a consequence of Theorem \ref{zerot1} and Corollary \ref{isot1lfd}.
\end{proof}

\section{Properties of the cohomology}

\subsection{Constructibility of the cohomology}

As we have seen in the previous section, the cohomology of the complex $\mathcal{C}^\bullet$ plays an important role in the theory of admissible deformations for a germ of a free divisor. From Schlessinger's Theorem 2.11 from \cite{schless}, we know that the main point in proving the existence of a hull is the finiteness of this cohomology. The following subsection is devoted to study this problem.

\begin{Definition} Let $X$ be a $n$-dimensional complex manifold. We denote by $\mathcal{D}_X$ the sheaf of differential operators on X and by $\mathcal{G}r_{F^\bullet}(\mathcal{D}_X)$ the sheaf on $T^*X$ of graded rings associated with the filtration $F^\bullet$ by the order of $\sigma(P)$ the principal symbol of a differential operator P.
\end{Definition}
\begin{Definition} Let $D\subset \C^n$ be a divisor defined by the ideal $I$. We define the $\mathcal{V}$-filtration relative to $D$ on $\mathcal{D}_{\C^n}$ by $$\mathcal{V}^D_k(\mathcal{D}_{\C^n}):=\{P\in \mathcal{D}_{\C^n}~|~P(I^j)\subset I^{j-k}~\forall j\in\Z \} $$ for all $k\in \Z$, where $I^j=\mathcal{O}_{\C^n}$ when $j$ is negative. Similarly, we define $$\mathcal{V}^D_k(\mathcal{D}_{\C^n,x}):=\{P\in \mathcal{D}_{\C^n,x}~|~P(f^j)\subset f^{j-k}~\forall j\in\Z \},$$ where $f$ is a local equation for $D$ at $x$. If there is no confusion, we denote $\mathcal{V}^D_k(\mathcal{D}_{\C^n})$ and $\mathcal{V}^D_k(\mathcal{D}_{\C^n,x})$ simply by $\mathcal{V}_k(\mathcal{D}_{\C^n})$ and $\mathcal{V}_k(\mathcal{D}_{\C^n,x})$, respectively.
\end{Definition}
\begin{Definition} A \emph{logarithmic differential operator} is an element of $\mathcal{V}_0(\mathcal{D}_{\C^n})$.
\end{Definition}
\begin{Remark} We have 
$$\Der(-\log D)=\Der_{\C^n}\cap \mathcal{V}_0(\mathcal{D}_{\C^n})=\mathcal{G}r^1_{F^\bullet}(\mathcal{V}_0(\mathcal{D}_{\C^n})),$$ and  
$$F^1(\mathcal{V}_0(\mathcal{D}_{\C^n}))=\mathcal{O}_{\C^n} \oplus \Der(-\log D).$$
\end{Remark}
\begin{proof} The first equality comes directly from the definitions. The second one is a consequence of the fact that $F^1(\mathcal{D}_{\C^n})=\mathcal{O}_{\C^n} \oplus \Der_{\C^n}$.\end{proof}
\begin{Definition} Let $\mathcal{M}$ be a $\mathcal{O}_{\C^n}$-module. A \emph{connection} on $\mathcal{M}$ with logarithmic poles along $D$ or a \emph{logarithmic connection} on $\mathcal{M}$, is a homomorphism over $\C$ $$\nabla\colon \mathcal{M} \to \Omega^1(\log D)\otimes \mathcal{M},$$ that verifies Leibniz's identity $$\nabla(hm)=dh\otimes m+h\nabla(m)$$ for any $h\in \mathcal{O}_{\C^n}$ and $m\in \mathcal{M}$, where $d$ is the exterior derivative over $\mathcal{O}_{\C^n}$. For any $q\in\mathbb{N}$ we will denote $\Omega^q(\log D)\otimes \mathcal{M}$ by $\Omega^q(\log D)(\mathcal{M})$.
\end{Definition}
\begin{Definition} Let $\mathcal{M}$ be a $\mathcal{O}_{\C^n}$-module with $\nabla$ a logarithmic connection. We can define the following left $\mathcal{O}_{\C^n}$-linear morphism $$\nabla'\colon\Der(-\log D)\to \mathcal{E}nd_{\C}(\mathcal{M})$$ $$\delta \mapsto \nabla_{\delta}$$ where $\nabla_{\delta}(m):=\langle\delta, \nabla(m)\rangle$.
\end{Definition}
\begin{Remark} The morphism $\nabla'$ verifies Leibniz's condition $$\nabla_{\delta}(hm)=\delta(h)m+h\nabla_{\delta}(m)$$ for any  $\delta\in\Der(-\log D)$, $h\in \mathcal{O}_{\C^n}$ and $m\in \mathcal{M}$.
\end{Remark}
\begin{Remark} Given  a left $\mathcal{O}_{\C^n}$-linear morphism $$\nabla'\colon\Der(-\log D)\to \mathcal{E}nd_{\C}(\mathcal{M})$$ verifying Leibniz's condition, we define $$\nabla\colon \mathcal{M}\to \Omega^1(\log D)(\mathcal{M}) $$ with $\nabla(m)$ the element of $\Omega^1(\log D)(\mathcal{M})=\mathcal{H}om_{\mathcal{O}_{\C^n}}(\Der(-\log D),\mathcal{M})$ such that $\nabla(m)(\delta)=\nabla'(\delta)(m)$.
\end{Remark}
\begin{Definition} A logarithmic connection $\nabla$ is \emph{integrable} if, for each $\delta,\delta'\in \Der(-\log D)$, it verifies
$$\nabla_{[\delta,\delta']}=[\nabla_{\delta},\nabla_{\delta'}] ,$$ where $[~,~]$ represents the Lie bracket in $\Der(-\log D)$ and the commutator in $\mathcal{E}nd_{\C}(\mathcal{M})$. 
\end{Definition}
\begin{Example}\label{connectder} Consider $\mathcal{M}=\Der_{\C^n}/\Der(-\log D)$. Then we can introduce on $\mathcal{M}$ the integrable logarithmic  connection defined by $\nabla_{\delta}:=[\delta,-]$. If we take $\mathcal{M}=\Der(-\log D)$ or $\Der_{\C^n}$, then $\nabla_{\delta}=[\delta,-]$ does not in general define a connection on $\mathcal{M}$ because it is not $\mathcal{O}_{\C^n}$-linear in $\delta$.
\end{Example}
\begin{Proposition}\label{equivintconmod}\emph{(\cite{calderon}, Corollary 2.2.6)} Let $D\subset\C^n$ be a free divisor and let $\mathcal{M}$ be a $\mathcal{O}_{\C^n}$-module. An integrable logarithmic connection on $\mathcal{M}$ gives rise to a left $\mathcal{V}_0(\mathcal{D}_{\C^n})$-module structure on $\mathcal{M}$ and vice versa.
\end{Proposition}
We now explain a condition that allows us to put a structure of $\mathcal{V}_0(\mathcal{D}_{\C^n})$-module on $\Der(-\log D)$ and $\Der_{\C^n}$.

Fix $D\subset\C^n$ a free divisor and $\delta_i=\sum_{j=1}^na_{ij}\partial/\partial x_j$, $i=1,\dots, n$ a basis for $\Der(-\log D)$, where $a_{ij}\in \mathcal{O}_{\C^n}$ for $i,j=1,\dots,n$. We know that $\Der(-\log D)$ forms a Lie subalgebra of $\Der_{\C^n}$, hence we can write $$[\delta_i,\delta_j]=\sum_{k=1}^nb_{jk}^i\delta_k$$ where $b_{jk}^{i}\in \mathcal{O}_{\C^n}$ for all $i,j,k=1,\dots,n$ and similarly we can write $$[\delta_i,\partial/\partial x_j]=\sum_{k=1}^nc_{jk}^i\frac{\partial}{\partial x_k}$$ where $c_{jk}^{i}\in \mathcal{O}_{\C^n}$ for all $i,j,k=1,\dots,n$. In this way we obtain the data of $2n$ matrices $B_i=(b_{jk}^i)$ and $C_i=(c_{jk}^i)$ of holomorphic function on $\C^n$. Let us write $\delta \cdot \partial:=[\delta,\partial]$ for any derivation $\partial$ and any logarithmic derivation $\delta$. Then we have
$$\delta_i \cdot \underline{\delta}^t=B_i\underline{\delta}^t, ~~1\le i \le n$$ and
$$\delta_i \cdot \underline{\partial}^t=C_i\underline{\partial}^t, ~~1\le i \le n$$ 
where $\underline{\delta}=(\delta_1,\dots,\delta_n)$ and $\underline{\partial}=(\partial/\partial x_1,\dots,\partial/\partial x_n)$.
\begin{Lemma}\label{lemintconder} For $i,j=1,\dots,n$ we have that $$\delta_i(C_j)-\delta_j(C_i)+[C_j,C_i]=\sum_{k=1}^n b^i_{jk}C_k$$ if and only if $$\sum_{k=1}^na_{kr}\frac{\partial(b^i_{jk})}{\partial x_l} =0, ~~\forall~ i,l,r=1,\dots,n.$$
\end{Lemma}
\begin{proof}We first notice that by definition $c^i_{jk}=-\partial(a_{ik})/\partial x_j$. The first equality is an equality between matrices, hence we can check it entry by entry. Let $1\ge l,r\ge 0$. We now check the entry $(l,r)$. In this case the expression becomes
$$-\delta_i(\frac{\partial(a_{jr})}{\partial x_l})+\delta_j(\frac{\partial(a_{ir})}{\partial x_l})+\sum_{k=1}^n \frac{\partial(a_{jk})}{\partial x_l} \frac{\partial(a_{ir})}{\partial x_k}+ $$ $$-\sum_{k=1}^n \frac{\partial(a_{ik})}{\partial x_l} \frac{\partial(a_{jr})}{\partial x_k}= -\sum_{k=1}^n b_{jk}^{i}\frac{\partial(a_{kr})}{\partial x_l}.$$
Consider now the Jacobi identity
$$[[\delta_i,\delta_j],\frac{\partial}{\partial x_l}]+[[\delta_j,\frac{\partial}{\partial x_l}],\delta_i]+[[\frac{\partial}{\partial x_l},\delta_i],\delta_j]=0. $$
The coefficient of $\partial/\partial x_r$ of the previous expression is
$$\delta_i(\frac{\partial(a_{jr})}{\partial x_l})-\delta_j(\frac{\partial(a_{ir})}{\partial x_l})-\sum_{k=1}^n \frac{\partial(a_{jk})}{\partial x_l} \frac{\partial(a_{ir})}{\partial x_k}+ $$ $$+\sum_{k=1}^n \frac{\partial(a_{ik})}{\partial x_l} \frac{\partial(a_{jr})}{\partial x_k} - \sum_{k=1}^n b_{jk}^{i}\frac{\partial(a_{kr})}{\partial x_l}-\sum_{k=1}^n a_{kr}\frac{\partial(b_{jk}^{i})}{\partial x_l}=0.$$
Hence, the first equality is satisfied if and only if
$$\sum_{k=1}^n a_{kr}\frac{\partial(b_{jk}^{i})}{\partial x_l}=0.$$
\end{proof}
\begin{Proposition}\label{intconder} We can define a structure of left $\mathcal{V}_0(\mathcal{D}_{\C^n})$-module on $\Der_{\C^n}$ if 
$$\sum_{k=1}^na_{kr}\frac{\partial(b^i_{jk})}{\partial x_l} =0, ~~\forall~i, l,r=1,\dots,n.$$
\end{Proposition}
\begin{proof} To define a structure of left $\mathcal{V}_0(\mathcal{D}_{\C^n})$-module on $\Der_{\C^n}$, we define the action of $\delta_i$ on any derivation $\partial$ by $$\delta_i \bullet\partial:=[\delta_i,\partial],$$ or in other words $$\delta_i \bullet \underline{\partial}^t:=C_i\underline{\partial}^t, ~~1\le i\le n. $$ The structure just introduced is a $\mathcal{V}_0(\mathcal{D}_{\C^n})$-module structure if and only if $$(\delta_i\delta_j-\delta_j\delta_i) \bullet \underline{\partial}^t=(\sum_{k=1}^nb^i_{jk}\delta_k) \bullet \underline{\partial}^t.$$ An easy computation shows us that this is true if and only if $$\delta_i(C_j)-\delta_j(C_i)+[C_j,C_i]=\sum_{k=1}^n b^i_{jk}C_k$$ hence we can conclude by Lemma \ref{lemintconder}.
\end{proof}
\begin{Remark} Notice that the action on $\Der_{\C^n}$ of any logarithmic derivation $\delta=\sum_{k=1}^n\beta_k\delta_k$ is given by
$$\delta\bullet  \underline{\partial}^t=\sum_{k=1}^n\beta_kC_k\underline{\partial}^t.$$
\end{Remark}
\begin{Lemma} For $i,j=1,\dots,n$, 
$$\sum_{k=1}^na_{lk}\frac{\partial(b^i_{jr})}{\partial x_k} =0, ~~\forall~ i,l,r=1,\dots,n$$ if and only if $$\delta_i(B_j)-\delta_j(B_i)+[B_j,B_i]=\sum_{k=1}^n b^i_{jk}B_k.$$
\end{Lemma}
\begin{proof} This is similar to the proof of Lemma \ref{lemintconder}.
\end{proof}
\begin{Proposition}\label{intconderlog} We can define a structure of left $\mathcal{V}_0(\mathcal{D}_{\C^n})$-module on $\Der(-\log D)$ if $$\sum_{k=1}^na_{lk}\frac{\partial(b^i_{jr})}{\partial x_k} =0, ~~\forall~ i,l,r=1,\dots,n.$$
\end{Proposition}
\begin{proof} As the proof of Proposition \ref{intconder}.
\end{proof}
\begin{Corollary} Let $D\subset\C^n$ be a linear free divisor. Then $\Der_{\C^n}$ and $\Der(-\log D)$ are left $\mathcal{V}_0(\mathcal{D}_{\C^n})$-modules.
\end{Corollary}
\begin{proof} In this case $b^i_{jk}\in\C$ and so the two previous conditions are trivially fulfilled.
\end{proof}
\begin{Corollary} Let $D\subset\C^2$ be a free divisor defined by a weighted homogeneous equation. Then $\Der_{\C^2}$ and $\Der(-\log D)$ are left $\mathcal{V}_0(\mathcal{D}_{\C^2})$-modules.
\end{Corollary}
\begin{proof} Because $D$ is defined by $f$ a weighted homogenous equation and because $\Der(-\log D)$ is a free $\mathcal{O}_{\C^2}$-module of rank 2, then we can choose $\chi,\delta$ as a basis of $\Der(-\log D)$, where $\chi$ is an Euler vector field and $\delta(f)=0$. Then $[\chi,\delta]=\alpha\delta$, where $\alpha\in\C$ and so all the $b^i_{jk}\in\C$. Hence  the two previous conditions are trivially fulfilled.
\end{proof}
\begin{Definition} Define the complex
$$\Omega^\bullet(\log D)(\Der_{\C^n}/\Der(-\log D)) :=\Omega^\bullet(\log D)\otimes_{\mathcal{O}_{\C^n}}(\Der_{\C^n}/\Der(-\log D)) $$ with differentials $$\nabla^p\colon\Omega^p(\log D)(\Der_{\C^n}/\Der(-\log D))\to \Omega^{p+1}(\log D)(\Der_{\C^n}/\Der(-\log D)) $$ given by $$\nabla^p(\omega \otimes \delta):=d\omega \otimes \delta+(-1)^p\omega \wedge \nabla(\delta),$$ where $d$ is the usual exterior derivative on $\Omega^\bullet(\log D)$ and $\nabla(\delta)$ is the element of $\Omega^1(\log D)\otimes_{\mathcal{O}_{\C^n}}(\Der_{\C^n}/\Der(-\log D))$ such that $[\nu,\delta]=\nu\cdot \nabla(\delta)$ for all $\nu\in \Der(-\log D)$.
\end{Definition}
\begin{Theorem}\label{isomcompl} There is an isomorphism of complexes of sheaves of complex vector spaces between $\Omega^\bullet(\log D)(\Der_{\C^n}/\Der(-\log D))$ and $\mathcal{C}^\bullet$, defined by $$\gamma^p\colon \Omega^p(\log D)(\Der_{\C^n}/\Der(-\log D))\to \mathcal{C}^p$$  $$\gamma^p(\omega_1\wedge\cdots\wedge\omega_p\otimes\delta)(\delta_1\wedge\cdots\wedge\delta_p):=\det(\omega_i\cdot\delta_j)_{1\le i,j\le p}\delta .$$
\end{Theorem}
\begin{proof} Applying Theorem 3.2.1 from \cite{calderon} in our case, we deduce that there is an isomorphism $\psi^\bullet$ between the complex $$\Omega^\bullet(\log D)(\Der_{\C^n}/\Der(-\log D))$$ and the dual of the logarithmic Spencer complex 
$$\mathcal{H}om_{\mathcal{V}_0(\mathcal{D}_{\C^n})}(\mathcal{V}_0(\mathcal{D}_{\C^n})\otimes_{\mathcal{O}_{\C^n}}\bigwedge^\bullet\Der(-\log D),\Der_{\C^n}/\Der(-\log D)).$$ %where $\mathcal{D}_{\C^n}$ is the sheaf of differential operators on $\C^n$ and $\mathcal{V}_0(\mathcal{D}_{\C^n})=\{P\in \mathcal{D}_{\C^n}~|~P((f)^j)\subset (f)^j~\forall j\in\Z \}$. 
The isomorphism is defined locally by
$$\psi^p((\omega_1\wedge\cdots\wedge\omega_p)\otimes \delta)(P\otimes(\delta_1\wedge\cdots\wedge\delta_p)):=P \cdot \det(\omega_i\cdot \delta_j)_{1\le i,j \le p} \cdot \delta.$$ On the other hand, we can write $\psi^p=\lambda^p\circ \gamma^p$, where $\lambda^p$ is the isomorphism $$\lambda^p\colon \mathcal{C}^p\to \mathcal{H}om_{\mathcal{V}_0(\mathcal{D}_{\C^n})}(\mathcal{V}_0(\mathcal{D}_{\C^n})\otimes_{\mathcal{O}_{\C^n}}\bigwedge^p\Der(-\log D),\Der_{\C^n}/\Der(-\log D)),$$ defined by $$\lambda^p(\alpha)(P\otimes(\delta_1\wedge\cdots\wedge\delta_p)):=P\cdot\alpha(\delta_1\wedge\cdots\wedge\delta_p).$$ 
A direct computation shows that $\lambda^p$ commutes with the differentials and hence defines an isomorphism of complexes.

Hence, also $\gamma^p$ is an isomorphism and commutes with the differentials and therefore defines an isomorphism of complexes.
\end{proof}
\begin{Definition} Let $D\subset \C^n$ be a divisor. We say that $D$ is a \emph{Koszul free divisor} at $x$ if it is free at $x$ and if there exists a basis $\delta_1,\dots,\delta_n$ of $\Der_x(-\log D)$ such that the sequence of symbols $\sigma(\delta_1),\dots,\sigma(\delta_n)$ is regular in $\mathcal{G}r_{F^\bullet}(\mathcal{D}_{\C^n})_x$. If $D$ is a Koszul free divisor at every point, we simply say that it is a Koszul free divisor.
\end{Definition}
Notice that for a free divisor $D$, to be Koszul is equivalent to being holonomic in the sense of Definition 3.8 from \cite{saito}, i.e. the logarithmic stratification of $D$ is locally finite. See \cite{grmondsch}, Theorem 7.4.
\begin{Example}\label{examplekoszulfree}
\begin{enumerate}
\item \emph{(\cite{locqhomkosz}, Example 2.8, 3))} Each reduced divisor $D\subset\C^2$ is Koszul free.
\item The normal crossing divisor of Example \ref{ncrdiv} is Koszul free.
\item \emph{(\cite{locqhomkosz}, Example 2.8, 5))} Consider the free divisor $D=V(2^8z^3 - 2^7x^2z^2 + 2^4x^4z + 2^43^2xy^2z - 2^2x^3y^2 - 3^3y^4)\subset \C^3$ with Saito matrix
 \begin{equation*} A=[\delta_1,\delta_2,\delta_3]=
\begin{bmatrix}
		6y&4x^2-48z&2x\\
		8z-2x^2&12xy&3y\\
		-xy&9y^2-16xz&4z\\
\end{bmatrix}.
\end{equation*}
Then the sequence of symbols $\sigma(\delta_1),\sigma(\delta_2),\sigma(\delta_3)$ is regular in $\Gr_{F^\bullet}(\mathcal{D}_{\C^n})$.
\item  \emph{(\cite{locqhomkosz}, Example 4.2)} Consider the free divisor $D=V(xy(x+y)(y+xz))\subset \C^3$ with Saito matrix
 \begin{equation*} 
\begin{bmatrix}
		x&4x^2&0\\
		y&-y^2&0\\
		0&-z(x+y)&xz+y\\
\end{bmatrix}
\end{equation*}
Then $D$ is not Koszul free.
\end{enumerate}
\end{Example}
\begin{Theorem}\label{constructcohomol} Let $(D,0)\subset(\C^n,0)$ be a germ of a Koszul free divisor such that $\sum_{k=1}^na_{kl}\partial(b^i_{jk})/\partial x_r =0$ for $i,j,l,r=1,\dots,n$ and $\sum_{l=1}^na_{kl}\partial(b^i_{jr})/\partial x_l =0,$ for $i,j,k,r=1,\dots,n$. Then all $\mathcal{H}^i(\mathcal{C}^\bullet)$ are constructible sheaves of finite dimensional complex vector spaces.
\end{Theorem}
\begin{proof} Write $\mathcal{E}_0=\Der_{\C^n}/\Der(-\log D)$, $\mathcal{E}_1=\Der_{\C^n}$ and $\mathcal{E}_2=\Der(-\log D)$.  Using the assumptions, we deduce from Proposition \ref{intconder} and \ref{intconderlog}, that we can consider the short exact sequence $$0\to \mathcal{E}_2\to\mathcal{E}_1\to\mathcal{E}_0\to 0$$ as a resolution of the $\mathcal{V}_0(\mathcal{D}_{\C^n})$-module $\mathcal{E}_0$. By twisting with $\mathcal{O}_{\C^n}[D]$, we find another $\mathcal{V}_0(\mathcal{D}_{\C^n})$-resolution $$0\to \mathcal{E}_2[D]\to\mathcal{E}_1[D]\to\mathcal{E}_0[D]\to 0.$$ By \cite{caldnarv}, Proposition 1.2.3, the complexes $\mathcal{D}_{\C^n} \stackrel{L}{\otimes}_{\mathcal{V}_0(\mathcal{D}_{\C^n})}\mathcal{E}_i[D]$, for $i=1,2$ are concentrated in degree zero. Hence, we can compute the complex $\mathcal{D}_{\C^n} \stackrel{L}{\otimes}_{\mathcal{V}_0(\mathcal{D}_{\C^n})}\mathcal{E}_0[D]$ through the above resolution as $$\mathcal{D}_{\C^n} \otimes_{\mathcal{V}_0(\mathcal{D}_{\C^n})}\mathcal{E}_2[D]\to \mathcal{D}_{\C^n} \otimes_{\mathcal{V}_0(\mathcal{D}_{\C^n})}\mathcal{E}_1[D].$$ By \cite{caldnarv}, Proposition 1.2.3, the above complex is holonomic in each degree and we deduce that $\boldsymbol{R}\mathcal{H}om_{\mathcal{D}_{\C^n}}(\mathcal{O}_{\C^n},\mathcal{D}_{\C^n} \stackrel{L}{\otimes}_{\mathcal{V}_0(\mathcal{D}_{\C^n})}\mathcal{E}_0[D])$ is constructible. By Theorem \ref{isomcompl} and by noticing that the isomorphism of \cite{caldnarv}, Corollary 3.1.5, is true for any $\mathcal{V}_0(\mathcal{D}_{\C^n})$-module, we have the following isomorphisms
$$\mathcal{C}^\bullet\cong \Omega^\bullet(\log D)(\mathcal{E}_0)\cong \boldsymbol{R}\mathcal{H}om_{\mathcal{D}_{\C^n}}(\mathcal{O}_{\C^n},\mathcal{D}_{\C^n} \stackrel{L}{\otimes}_{\mathcal{V}_0(\mathcal{D}_{\C^n})}\mathcal{E}_0[D])$$
and hence, we can conclude.
\end{proof}
%\begin{Theorem} Let $D$ be a Koszul free divisor, then all $\mathcal{H}^i(\mathcal{C}^\bullet)$ are constructible sheaves of finite dimensional complex vector spaces.
%\end{Theorem}
%\begin{proof} Consider the short exact sequence: $$0\to \Der(-\log D) \to \Der_{\C^n}\to \Der_{\C^n}/\Der(-\log D)\to 0.$$ Because $D$ is free, for each $p$ we have that $\Omega^p(\log D)$ is a free $\mathcal{O}_{\C^n}$-module and hence, for every $p$, we have the following short exact sequence: $$0\to\Omega^p(\log D)\otimes_{\mathcal{O}_{\C^n}} \Der(-\log D) \to \Omega^p(\log D)\otimes_{\mathcal{O}_{\C^n}}\Der_{\C^n}\to$$ $$\to\Omega^p(\log D)\otimes_{\mathcal{O}_{\C^n}}\Der_{\C^n}/\Der(-\log D)\to 0.$$ This gives us a short exact sequence of complexes: $$0\to\Omega^\bullet(\log D)(\Der(-\log D)) \to \Omega^\bullet(\log D)(\Der_{\C^n})\to$$ $$\to\Omega^\bullet(\log D)(\Der_{\C^n}/\Der(-\log D))\to 0,$$ where the first two complexes are defined in the same way as the third. By Theorem \ref{isomcompl}, we obtain the following short exact sequence: $$0\to\Omega^\bullet(\log D)(\Der(-\log D)) \to \Omega^\bullet(\log D)(\Der_{\C^n})\to \mathcal{C}^\bullet \to 0.$$ By Proposition 1.2.3 and Corollary 3.1.6 of \cite{caldnarv}, we have that the first two complexes are perverse sheaves and hence $\mathcal{C}^\bullet$ is constructible. \end{proof}
\begin{Corollary} Let $(D,0)\subset(\C^n,0)$ be a germ of a Koszul free divisor such that $\sum_{k=1}^na_{kl}\partial(b^i_{jk})/\partial x_r =0$ and $\sum_{k=1}^na_{lk}\partial(b^i_{jr})/\partial x_k =0,$ for $i,j,l,r=1,\dots,n$. Then $\mathbf{FD}_D$ has a hull.
\end{Corollary}
\begin{proof} By Theorem \ref{constructcohomol}, condition (H3) from \cite{schless} is satisfied. Then the result follows from Theorem \ref{isdeffunct} and Theorem 2.11 from \cite{schless}.
\end{proof}
\begin{Corollary} Let $(D,0)\subset(\C^2,0)$ be a germ of a free divisor defined by a weighted homogeneous equation. Then $\mathbf{FD}_D$ has a hull.
\end{Corollary}
\begin{proof} By Example \ref{examplekoszulfree}, $(D,0)$ is Koszul. Because $(D,0)$ is defined by $f$ a weighted homogenous equation then we can choose $\chi,\delta$ as a basis of $\Der(-\log D)$, where $\chi$ is an Euler vector field and $\delta(f)=0$. Then $[\chi,\delta]=\alpha\delta$, where $\alpha\in\C$ and so all the $b^i_{jk}\in\C$. Hence all the hypothesis of previous Corollary are fulfilled.
\end{proof}
\begin{Corollary} Let $(D,0)\subset(\C^n,0)$ be a germ of a Koszul linear free divisor. Then all $\mathcal{H}^i(\mathcal{C}^\bullet)$ are constructible sheaves of finite dimensional complex vector spaces. 
\end{Corollary}
\begin{proof} This follows from Theorem \ref{constructcohomol} and the fact that if $(D,0)$ is linear then $b^i_{jk}\in\C$. 
\end{proof}
%\begin{Corollary} Let $(D,0)\subset(\C^n,0)$ be a germ of a Koszul linear free divisor, then $\mathbf{FD}_D$ has a hull.
%\end{Corollary}
The author is not aware if there exists a subclass of the Koszul free divisor that fulfil the assumptions of Theorem \ref{constructcohomol}. However, we know that not all Koszul free divisor satisfies them. A direct computation shows that the last Koszul free divisor described in Example \ref{examplekoszulfree} does not fulfil them.

Moreover, the author thinks that the approach used to put a logarithmic connection on $\Der_{\C^n}$ and $\Der(-\log D)$ is a particular case of the notion of integrability up to homotopy, see \cite{abad2009representations}.

\subsection{Propagation of Deformations}
In this final subsection, we prove a result which highlights the difference between the theory of admissible deformations and the classical deformation theory of singularities.

We suppose that $(D,0)\subset (\C^n,0)$ is a germ of a free divisor such that there exists a germ of a free divisor $(D',0)\subset (\C^{n-1},0)$ such that $(D,0)=(D' \times \C,0)$, i.e. there exists a defining equation for $D$ in $\C[[x_1,\dots,x_{n-1}]]$.

\begin{Theorem}\emph{(Corollary \ref{propdefisocohom})} There is an isomorphism of sheaves $$\pi^{-1}\mathcal{H}^i(\mathcal{C}_{D'}^\bullet)\cong\mathcal{H}^i(\mathcal{C}_D^\bullet) $$
where $\pi\colon (D,0)\to (D',0)$ is the projection on the first factor of $(D,0)=(D'\times \C,0)$. In particular, we have $$\pi^{-1}\mathcal{FT}^1(D')\cong\mathcal{FT}^1(D).$$
\end{Theorem}
Observe that in the ordinary deformation theory of singularities, if  $(D,0)=(D' \times \C,0)$ and $T^1_{(D',0)}$ is non-zero then $T^1_{(D,0)}$ is infinite dimensional. See \cite{intdef}, Chapter II, 1.4. 

\begin{Lemma} In this situation $$\Der(-\log D)=(\Der(-\log D')\otimes_{\mathcal{O}_{\C^{n-1},0}}\mathcal{O}_{\C^n,0})\oplus \mathcal{O}_{\C^n,0}\frac{\partial}{\partial x_n}$$ and $$\Der_{\C^n}/\Der(-\log D)=\Der_{\C^{n-1}}/\Der(-\log D')\otimes_{\mathcal{O}_{\C^{n-1},0}}\mathcal{O}_{\C^n,0}.$$ Hence, if $\delta \in \Der(-\log D)$, it can be written as $\delta=(\delta',h\partial / \partial x_n)$, where $\delta' \in \Der(-\log D')\otimes_{\mathcal{O}_{\C^{n-1},0}}\mathcal{O}_{\C^n,0}$ and $h\in \mathcal{O}_{\C^n,0}$.
\end{Lemma}
%In general, we have the following result that will be useful in what follows:
%\begin{Lemma} Let $R$ be a commutative ring and $A$ and $B$ two R-modules, then: $$\bigwedge^p(A\oplus B)= \bigoplus_{i+j=p}(\bigwedge^i A \otimes \bigwedge^j B).$$
%\end{Lemma}
To distinguish between the complexes for $(D,0)$ and for $(D',0)$ we will denote them respectively by $(\mathcal{C}^\bullet_{D},d^\bullet_{D})$ and $(\mathcal{C}^\bullet_{D'},d^\bullet_{D'})$.
\begin{Proposition}\label{isopropdefcomplex} There is an isomorphism $$\varrho\colon \bigwedge^p \Der(-\log D)\to (\mathcal{O}_{\C^n,0} \otimes_{\mathcal{O}_{\C^{n-1},0}}\bigwedge^p \Der(-\log D'))$$ $$\oplus(\mathcal{O}_{\C^n,0} \otimes_{\mathcal{O}_{\C^{n-1},0}}\bigwedge^{p-1} \Der(-\log D'))$$ $$(\delta_1 \wedge \cdots\wedge \delta_p)=(\delta'_1,h_1\frac{\partial}{\partial x_n})\wedge \cdots \wedge (\delta'_p,h_p\frac{\partial}{\partial x_n}) \mapsto $$ $$(\delta'_1\wedge \cdots \wedge \delta'_p,\sum_{k=1}^p (-1)^{p-k}h_k\delta'_1\wedge \cdots \wedge \widehat{\delta'_k} \wedge \cdots \wedge \delta'_p) $$
\end{Proposition}
\begin{proof} This is because $\bigwedge^p \mathcal{O}_{\C^n,0}=0$ for $p \ge 2$ and because in general, if $R$ is a commutative ring and $A$ and $B$ are $R$-modules, then $$\bigwedge^p(A\oplus B)= \bigoplus_{i+j=p}(\bigwedge^i A \otimes_R \bigwedge^j B).$$ \end{proof}
As a consequence
\begin{Corollary} With the hypotheses of Proposition \ref{isopropdefcomplex} $$\mathcal{C}^p_D=\mathcal{H}om_{\mathcal{O}_{\C^n,0}}(\mathcal{O}_{\C^n,0}\otimes_{\mathcal{O}_{\C^{n-1},0}}\bigwedge^p\Der(-\log D'), \Der_{\C^n}/\Der(-\log D))$$ $$\oplus  \mathcal{H}om_{\mathcal{O}_{\C^n}}(\mathcal{O}_{\C^n,0}\otimes_{\mathcal{O}_{\C^{n-1},0}}\bigwedge^{p-1}\Der(-\log D'), \Der_{\C^n}/\Der(-\log D))=$$ $$=\mathcal{C}^p_{D'}\otimes_{\mathcal{O}_{\C^{n-1},0}}\mathcal{O}_{\C^n,0} \oplus \mathcal{C}^{p-1}_{D'}\otimes_{\mathcal{O}_{\C^{n-1},0}}\mathcal{O}_{\C^n,0}.$$
\end{Corollary}
\begin{Remark}\label{element} It is possible to write an element $\Gamma\in\mathcal{C}^p_D$ for $p>0$ as $$\Gamma=(\psi,\phi)=(\sum_{i\ge 0}x_n^i\psi_i,\sum_{i\ge 0}x_n^i\phi_i)=\sum_{i\ge 0}x_n^i(\psi_i,\phi_i)$$ with $\psi_i\in\mathcal{C}^p_{D'}$ and $\phi_i\in\mathcal{C}^{p-1}_{D'}$.

%$\psi \in \mathcal{H}om_{\mathcal{O}_{\C^n}}(\mathcal{O}_{\C^n}\otimes_{\mathcal{O}_{\C^{n-1}}}\bigwedge^p\Der(-\log D'), \Der_{\C^n}/\Der(-\log D))$ as: $$\psi = \sum_{i\ge 0}x_n^i\psi_i : \mathcal{O}_{\C^n}\otimes_{\mathcal{O}_{\C^{n-1}}}\bigwedge^p\Der(-\log D') \to \Der_{\C^n}/\Der(-\log D)$$ $$g\otimes(\sigma_1\wedge \cdots \wedge \sigma_p) \mapsto  g\cdot \sum_{i\ge 0}x_n^i\psi_i(\sigma_1\wedge \cdots \wedge \sigma_p)$$ where $\psi_i\in\mathcal{H}om_{\mathcal{O}_{\C^n}}(\bigwedge^p\Der(-\log D'), \Der_{\C^n}/\Der(-\log D)).$
\end{Remark}
By Remark \ref{element}, we can describe the differentials
\begin{Corollary} The differentials have the following expression $$\tilde{d}^p \colon \mathcal{H}om_{\mathcal{O}_{\C^n,0}}(\mathcal{O}_{\C^n,0}\otimes_{\mathcal{O}_{\C^{n-1},0}}\bigwedge^p\Der(-\log D'), \Der_{\C^n}/\Der(-\log D)) \to$$ $$\to \mathcal{H}om_{\mathcal{O}_{\C^n,0}}(\mathcal{O}_{\C^n,0}\otimes_{\mathcal{O}_{\C^{n-1},0}}\bigwedge^{p+1}\Der(-\log D'), \Der_{\C^n}/\Der(-\log D))$$ where $$( \tilde{d}^p(\psi))(\sigma_1 \wedge \cdots \wedge \sigma_{p+1}):=\sum_{i\ge 0}x_n^i(d^p_{D'}(\psi_i))(\sigma_1 \wedge \cdots \wedge \sigma_{p+1}).$$
\end{Corollary}
\begin{Proposition} The differential on $\mathcal{C}^\bullet_{D}$ is given by $$d^p_D\colon \mathcal{C}^p_{D} \to \mathcal{C}^{p+1}_{D} $$ $$(\psi, \phi) \mapsto (\tilde{d}^p(\psi),\tilde{d}^{p-1}(\phi)+(-1)^{p+1}[\frac{\partial}{\partial x_n},\psi(-)]).$$
\end{Proposition}
\begin{proof} Consider $\Gamma=(\psi,\phi)\in \mathcal{C}^p_D$. %and $\Sigma=(\sigma_1\wedge \cdots \wedge \sigma_p)\in \bigwedge^p\Der(-\log D)$. We can represent $\Sigma$ as : $$\Sigma=((\sigma'_1+h_1\partial / \partial x_n)\wedge \cdots \wedge (\sigma'_p+h_p\partial / \partial x_n))$$ with $\sigma'_i \in \Der(-\log D')\otimes_{\mathcal{O}_{\C^{n-1}}}\mathcal{O}_{\C^n}$ and $h_i \in \mathcal{O}_{\C^n}$. Then we have that: $$\Gamma(\Sigma)=\psi(\sigma'_1\wedge \cdots \wedge \sigma'_p)+\sum_{k=1}^p (-1)^{p-k}h_k\phi(\sigma'_1\wedge \cdots \wedge \widehat{\sigma'_k} \wedge \cdots \wedge \sigma'_p).$$ 
We want now to compute $d^p_D(\Gamma)\in \mathcal{C}^{p+1}_D$. By Remark \ref{element}, we need to check it only on $\Der(-\log D')$, hence 
\begin{equation*}
(d^p_D(\Gamma))(\sigma_1\wedge \cdots \wedge \sigma_{p+1}, \nu_1 \wedge \cdots \wedge \nu_p)= \end{equation*}
\begin{equation}\label{cohomologin2part}= (d^p_D(\Gamma))(\sigma_1\wedge \cdots \wedge \sigma_{p+1})+ (d^p_D(\Gamma))(\nu_1\wedge \cdots \wedge \nu_p),
\end{equation}
where $\sigma_i,\nu_j\in\Der(-\log D')$.

 We now look at the first part of the right hand side of the previous equality $$(d^p_D(\Gamma))(\sigma_1\wedge \cdots \wedge \sigma_{p+1})=$$ $$=\sum_{i=1}^{p+1}(-1)^i[\sigma_i,\Gamma(\sigma_1 \wedge \dots \wedge \widehat{\sigma_i} \wedge \dots \wedge \sigma_{p+1})]  +$$ $$+\\ \sum_{1 \le i<j \le p+1}(-1)^{i+j-1}\Gamma([\sigma_i,\sigma'_j]\wedge \sigma_1 \wedge \dots \wedge \widehat{\sigma_i}\wedge \dots \wedge \widehat{\sigma_j}\wedge \dots \wedge \sigma_{p+1})=$$ $$=(\tilde{d}^p(\psi))(\sigma_1\wedge \cdots \wedge \sigma_{p+1}).$$ Consider now the second part of equation \eqref{cohomologin2part} from above. Put $\nu_{p+1}:=\partial / \partial x_n$. Then we have $$\varrho((\nu_1,0\cdot\frac{\partial}{\partial x_n})\wedge \cdots \wedge(\nu_p,0\cdot\frac{\partial}{\partial x_n})\wedge (0,1\cdot\frac{\partial}{\partial x_n}))=(0,\nu_1\wedge \cdots \wedge \nu_p) $$and hence $$(d^p_D(\Gamma))(\nu_1\wedge \cdots \wedge \nu_p)=(d^p_D(\Gamma))(\nu_1\wedge \cdots \wedge \nu_{p+1}) =$$ $$=\sum_{i=1}^{p+1}(-1)^i[\nu_i,\Gamma(\nu_1 \wedge \dots \wedge \widehat{\nu_i} \wedge \dots \wedge \nu_{p+1})]  +$$ $$+\\ \sum_{1 \le i<j \le p+1}(-1)^{i+j-1}\Gamma([\nu_i,\nu_j]\wedge \nu_1 \wedge \dots \wedge \widehat{\nu_i}\wedge \dots \wedge \widehat{\nu_j}\wedge \dots \wedge \nu_{p+1}).$$  We now look at the first part of the right hand side of the previous equality, it is equal to $$\sum_{i=1}^{p+1}(-1)^i[\nu_i,\Gamma(\nu_1 \wedge \dots \wedge \widehat{\nu_i} \wedge \dots \wedge \nu_{p+1})]  =$$ $$=(-1)^{p+1}[\frac{\partial}{\partial x_n}, \Gamma(\nu_1 \wedge \cdots \wedge \nu_p)]+\sum_{i=1}^{p}(-1)^i[\nu_i,\Gamma(\nu_1 \wedge \dots \wedge \widehat{\nu_i} \wedge \dots \wedge\nu_p \wedge \frac{\partial}{\partial x_n})]  =$$ $$=(-1)^{p+1}[\frac{\partial}{\partial x_n}, \psi(\nu_1 \wedge \cdots \wedge \nu_p)]+\sum_{i=1}^{p}(-1)^i[\nu_i,\phi(\nu_1 \wedge \dots \wedge \widehat{\nu_i} \wedge \dots \wedge\nu_p)].$$ Moreover $$ \sum_{1 \le i<j \le p+1}(-1)^{i+j-1}\Gamma([\nu_i,\nu_j]\wedge \nu_1 \wedge \dots \wedge \widehat{\nu_i}\wedge \dots \wedge \widehat{\nu_j}\wedge \dots \wedge \nu_{p+1})=$$ $$\sum_{i=1}^p (-1)^{i+p} \Gamma([\nu_i,\frac{\partial}{\partial x_n}],\wedge \nu_1 \wedge \cdots \wedge \widehat{\nu_i} \wedge \cdots \wedge \nu_p)+$$ $$+\\ \sum_{1 \le i<j \le p}(-1)^{i+j-1}\Gamma([\nu_i,\nu_j]\wedge \nu_1 \wedge \dots \wedge \widehat{\nu_i}\wedge \dots \wedge \widehat{\nu_j}\wedge \dots \wedge \nu_p\wedge \frac{\partial}{\partial x_n})=$$ $$=\\ \sum_{1 \le i<j \le p}(-1)^{i+j-1}\phi([\nu_i,\nu_j]\wedge \nu_1 \wedge \dots \wedge \widehat{\nu_i}\wedge \dots \wedge \widehat{\nu_j}\wedge \dots \wedge \nu_p).$$ Here the important point is that $[\nu_i,\partial / \partial x_n]=0$ because $\nu_i\in \Der(-\log D')$. Hence $$(d^p_D(\Gamma))(\nu_1\wedge \cdots \wedge \nu_p)=$$ $$=(-1)^{p+1}[\frac{\partial}{\partial x_n}, \psi(\nu_1 \wedge \cdots \wedge \nu_p)]+\sum_{i=1}^{p}(-1)^i[\nu_i,\phi(\nu_1 \wedge \dots \wedge \widehat{\nu_i} \wedge \dots \wedge\nu_p)] +$$ $$+\\ \sum_{1 \le i<j \le p}(-1)^{i+j-1}\phi([\nu_i,\nu_j]\wedge \nu_1 \wedge \dots \wedge \widehat{\nu_i}\wedge \dots \wedge \widehat{\nu_j}\wedge \dots \wedge \nu_p)=$$ $$=(-1)^{p+1}[\frac{\partial}{\partial x_n}, \psi(\nu_1 \wedge \cdots \wedge \nu_p)]+(\tilde{d}^{p-1}(\phi))(\nu_1\wedge \cdots \wedge \nu_p).$$ \end{proof}
\begin{Proposition} We can rewrite the differential as $$d^p_D\colon \mathcal{C}^p_{D} \to \mathcal{C}^{p+1}_{D} $$ $$(\psi, \phi) \mapsto \sum_{i \ge 0} x_n^i (d^p_{D'}(\psi_i),d^{p-1}_{D'}(\phi_i)+(-1)^{p+1}(i+1)\psi_{i+1}).$$
\end{Proposition}
\begin{Definition} We define the morphism $J$ to be the inclusion $$J\colon\mathcal{H}om_{\mathcal{O}_{\C^{n-1},0}}(\bigwedge^p\Der(-\log D'), \Der_{\C^{n-1}}/\Der(-\log D'))=\mathcal{C}_{D'}^p \hookrightarrow$$ $$\mathcal{C}_D^p=\mathcal{H}om_{\mathcal{O}_{\C^n,0}}(\mathcal{O}_{\C^n,0}\otimes_{\mathcal{O}_{\C^{n-1},0}}\bigwedge^p\Der(-\log D'), \Der_{\C^n}/\Der(-\log D))\\ \oplus$$ $$\mathcal{H}om_{\mathcal{O}_{\C^n,0}}(\mathcal{O}_{\C^n,0}\otimes_{\mathcal{O}_{\C^{n-1},0}}\bigwedge^{p-1}\Der(-\log D'), \Der_{\C^n}/\Der(-\log D))=$$ $$=\mathcal{C}^p_{D'}\otimes_{\mathcal{O}_{\C^{n-1}}}\mathcal{O}_{\C^n,0} \oplus \mathcal{C}^{p-1}_{D'}\otimes_{\mathcal{O}_{\C^{n-1},0}}\mathcal{O}_{\C^n,0}$$ $$ \psi \mapsto x_n^0(\psi,0).$$
\end{Definition}
All the previous work was devoted  proving that in order to compute the cohomology of $D$ it is enough to compute that of $D'$:  
\begin{Theorem} The morphism $J$ is a quasi-isomorphism.
\end{Theorem}
\begin{proof} It is enough to show that the cokernel of $J$ is acyclic. Consider then $$\Gamma= \sum_{i \ge 1}x_n^i(\psi_i, \phi_i) + (0,\phi_0) \in \coker(J)$$ and suppose that $\Gamma \in \ker(d^p_D)$. Then we have $d^p_{D'}(\psi_i)=0$ and $d^p_{D'}(\phi_i)=(-1)^p(i+1)\psi_{i+1}$ for all $i\ge 0$. Define now $$\Lambda := \sum_{i\ge 1}x_n^i(\frac{(-1)^p\phi_{i-1}}{i},0)\in \mathcal{C}^{p-1}_D$$ then we have that $$d^{p-1}_D(\Lambda)=\sum_{i\ge 1}x_n^i(d^{p-1}_{D'}\Lambda_i,(-1)^p (i+1)\Lambda_{i+1})=\Gamma.$$ Hence, $\Gamma$ vanishes in cohomology. \end{proof}
\begin{Corollary}\label{propdefisocohom} There is an isomorphism of sheaves $$\pi^{-1}\mathcal{H}^i(\mathcal{C}_{D'}^\bullet)\cong\mathcal{H}^i(\mathcal{C}_D^\bullet) $$
where $\pi\colon (D,0)\to (D',0)$ is the projection on the first factor of $(D,0)=(D'\times \C,0)$. In particular, we have $$\pi^{-1}\mathcal{FT}^1(D')\cong\mathcal{FT}^1(D).$$
\end{Corollary}
\begin{proof} This follows because $\pi^{-1}$ is an exact functor.
\end{proof}

\bibliography{bibliothesis}{}
\bibliographystyle{plain}

\end{document}